\documentclass[smallextended,numbook,runningheads]{svjour3}     
\smartqed  
\usepackage{mathptmx}      
\usepackage[utf8]{inputenc}             
\usepackage[english]{babel}             
\usepackage[noadjust]{cite}             
\usepackage{graphicx}                   
\usepackage{amsmath,amssymb,amsfonts,mathrsfs} 
\usepackage{url}                        
\usepackage{nameref}                     
\usepackage[usenames,dvipsnames]{color}  
\usepackage[pdftex,bookmarks=false, colorlinks=true, linkcolor=black, urlcolor=black, citecolor=black, linktoc=section]{hyperref} 
\usepackage{epstopdf}                    
\usepackage[ruled]{algorithm2e}         

\newcommand\matlab{\textsc{Matlab}}

\DeclareMathOperator{\trace}{trace}
\DeclareMathOperator{\orth}{\texttt{orth}}

\newcommand{\argmax}{\text{arg}\max}
\newcommand{\conj}[1]{\bar{#1}}
\DeclareMathOperator{\im}{Im}
\DeclareMathOperator{\re}{Re}

\DeclareMathOperator{\Lop}{\mathscr{L}}

\DeclareMathOperator{\H2}{\mathcal{H}_2}

\DeclareMathOperator{\vspan}{Span}
\DeclareMathOperator{\vrange}{Range}
\newcommand{\R}{\mathcal{R}}

\newcommand{\Vbase}{V}

\DeclareMathOperator{\Kspace}{\mathcal{K}}

\DeclareMathOperator{\kron}{\otimes}

\newcommand{\RR}{\mathbb{R}}
\newcommand{\tol}{\texttt{tol}}

\journalname{}
\begin{document}

\title{Residual-based iterations for the generalized Lyapunov~equation}

\author{Tobias Breiten \and 
Emil Ringh
}

\institute{Tobias Breiten \at
              Institute for Mathematics and Scientific Computing, Karl-Franzens-Universit{\"a}t\\ 
              Heinrichstrasse 36, 8010 Graz, Austria\\
              \email{tobias.breiten@uni-graz.at}
           \and
           Emil Ringh \at
              Department of Mathematics, KTH Royal Institute of Technology\\
              Lindstedtsv{\"a}gen 25, 100 44 Stockholm, Sweden\\
              \email{eringh@kth.se}
}

\maketitle

\begin{abstract}
This paper treats iterative solution methods to the generalized Lyapunov equation. 
Specifically it expands the existing theoretical justification for the alternating linear scheme (ALS) from the stable Lyapunov equation to the stable generalized Lyapunov equation. 
Moreover, connections between the energy-norm minimization in ALS and the theory to H2-optimality of an associated bilinear control system are established.
It is also shown that a certain ALS-based iteration can be seen as iteratively constructing rank-1 model reduction subspaces for bilinear control systems associated with the residual. 
Similar to the ALS-based iteration, the fixed-point iteration can also be seen as a residual-based method minimizing an upper bound of the associated energy norm. 
Lastly a residual-based generalized rational-Krylov-type subspace is proposed for the generalized Lyapunov equation.


\end{abstract}

\section{Introduction}\label{sect:intro}
This paper concerns iterative ways to compute approximate solutions to what has become known as  the \emph{generalized Lyapunov equation}
\begin{align}\label{eq:GenLyap}
\Lop(X) + \Pi(X) + BB^T = 0,
\end{align}
where $X\in\RR^{n\times n}$ is unknown, and the operators $\Lop, \Pi: \RR^{n\times n} \rightarrow \RR^{n\times n}$ are defined as
\begin{align}
\Lop(X) := AX + XA^T \label{eq:L}\\
\Pi(X) := \sum_{i=1}^m N_iXN_i^T, \label{eq:Pi}
\end{align}
with $A, N_i\in\RR^{n\times n}$ for $i=1,\dots,m$ given.
The operator $\Lop$ is commonly known as the \emph{Lyapunov operator}, and $\Pi$ is sometimes called a \emph{correction}.
We further assume that $A$ is \emph{stable}, i.e., $A$ has all its eigenvalues in the left-half plane, which implies that $\Lop$ is invertible \cite[Theorem~4.4.6]{Horn:1991:MATAN2}. Moreover, we assumed that $\rho(\Lop^{-1}\Pi)<1$, where $\rho$ denotes the (operator) spectral radius. The assumption on the spectral radius implies that \eqref{eq:GenLyap} has a unique solution, see, e.g., \cite[Theorem~2.1]{jarlebring2017krylov}. Furthermore, the definition of $\Pi$ in \eqref{eq:Pi} implies that it is non-negative, in the sense that $\Pi(X)$ is positive semidefinite when $X$ is positive semidefinite. Thus one can assert that the unique solution $X$ is indeed positive definite, see, e.g., \cite[Theorem~3.9]{DammBenner} or \cite[Theorem~4.1]{DammDirectADI}.

\subsection{Related work}
The standard Lyapunov equation, $AX+XA^T+BB^T = 0$, has been well studied for a long time and considerable research effort has been, and is still, put into finding efficient algorithms for computing the solution.
Some examples of methods are the classical  Bartels-Stewart algorithm \cite{Bartels1972} for small and dense problems based on factorizing the matrix $A$ and backward substitution, the Smith method presented in \cite{Smith:1968:Matrix}, and the Riemannian optimization method \cite{vandereycken2010riemannian} which computes a low-rank approximation by minimizing an associated cost function over the manifold of rank-$k$ matrices, where $k\ll n$.
The Lyapunov equation has a lose connection to control theoretic methods,
such as the iterative rational Krylov algorithm (IRKA) \cite{Gugercin2008,Flagg:2012:Convergence} which computes locally $\H2$-optimal reduced order systems.
Related research is presented in a series of paper \cite{Druskin2009,Druskin2010,Druskin.Simoncini.11}, where Druskin and co-authors develop a strategy to choose shifts for the rational Krylov subspace for efficient subspace reduction when solving PDEs \cite{Druskin2009,Druskin2010}, as well as for model reduction of linear single-input-single-output (SISO) systems and solutions to Lyapunov equations \cite{Druskin.Simoncini.11}. Instead of computing full spaces iteratively with a method such as IRKA, the idea is to construct an infinite sequence with asymptotically optimal convergence speed \cite{Druskin2009}. Then the subspace can be dynamically extended as needed, until required precision is achieved, rather than starting the process by deciding the dimension of the space, an a priori parameter choice with (complex) implications of the final approximation accuracy. The idea is also further developed by using tangential directions, proving especially useful for situations where the right-hand side is not of particularly low rank \cite{Druskin:2014:AdaptiveTangetial}, e.g., multiple-input-multiple-output (MIMO) systems.
For a more complete overview of results and techniques for Lyapunov equations, see, e.g., the review article \cite{Simoncini:2016:Computational}.

The generalized Lyapunov equation \eqref{eq:GenLyap} has received increased attention over the last decade. Results on low-rank approximability has emerged \cite[Theorem~1]{Benner:2013:low}\cite[Theorem~2]{jarlebring2017krylov}, i.e., similarly to the standard Lyapunov equation one can in certain cases when the right-hand-side $B$ is of low rank, $r\ll n$, expect the singular values of the solution to decay rapidly even for the generalized Lyapunov equation. Such low-rank approximability is useful for large-scale problems since algorithms can be adapted to exploit the low-rank format, reducing computational effort and storage requirement. Example of algorithms exploiting low-rank structures are a Bilinear ADI method \cite{Benner:2013:low}, specializations of Krylov methods for matrix equations \cite{jarlebring2017krylov}, as well as greedy low-rank methods \cite{Kressner:2015:Truncated}, and exploitations of the fixed-point iteration \cite{Shank2016}.
Through the connection with bilinear control systems there is an extension of IRKA, known as bilinear iterative rational Krylov (BIRKA) \cite{Benner2012,Flagg2015}.
There are also methods based on Lyapunov and ADI-preconditioned GMRES and BICGSTAB \cite{DammDirectADI}, and in general for problems with tensor product structure \cite{kressner2010krylov}.
We also mention that in the case when the correction $\Pi$ has low operator-rank, typically if all the matrices $N_i$ are of low rank, there is a specialization of the Sherman-Morrison-Woodbury formula to the linear matrix equation, see, e.g. \cite[Section~3]{DammDirectADI}.

\subsection{Outline and summary of contributions}
The paper is outlined as follows: In Section~\ref{sect:prel} we present some existing theory and preliminary results which sets the context of the paper. This is followed up in Section~\ref{sect:H2_connection} where we prove that the alternating linear scheme (ALS) presented by Kressner and Sirkovi{\'c} in \cite{Kressner:2015:Truncated} computes search directions which at each step fulfill a first order necessary condition for being $\H2$-optimal, in a certain sense. Moreover, we show equivalence between ALS and BIRKA. In Section~\ref{sect:fixed-point} we make an analogue to the fixed-point iteration, showing that it minimizes an upper bound of the energy-norm,  before we in Section~\ref{sect:Res_Rat_Kry} present a residual-based generalized rational-Krylov-type subspace adapted for solving the generalized Lyapunov equation. While the main focus of this paper is  on solution methods to the generalized Lyapunov equation, in Section~\ref{sect:lin} we draw parallels with rational Krylov subspaces for the standard Lyapunov equation. We end the paper with some numerical experiments presented in Section~\ref{sect:num}, and conclusions and outlooks in Section~\ref{sect:concl}.

\section{Preliminaries}\label{sect:prel}
\subsection{Generalized matrix equations and approximations}
In this paper we are concerned with iterative methods for computing approximative solutions to the generalized Lyapunov equation \eqref{eq:GenLyap}. In this section we recall some basic definitions and results that will be used in later in the paper.
Regarding the generalized Lyapunov equation there is one special class of problems where more can be said, and that is the symmetric case.
\begin{definition}[Symmetric generalized Lyapunov equation]\label{def:sym_lyap}
We call equation \eqref{eq:GenLyap} a \emph{symmetric} generalized Lyapunov equation if $A = A^T$ and $N_i=N_i^T$ for $i=1,\dots,m$.
\end{definition}

In general we will think of $\hat X_k\in \RR^{n\times n}$ as an approximation of the solution to \eqref{eq:GenLyap}, where $k$ is an iteration count. The goal is to find an $\hat X_k$ such that $\|X-\hat X_k\|$ is small for some norm, where $X$ is the exact solution to \eqref{eq:GenLyap}. However, in many situations there is no precise definition, and no need for one, of what is meant by approximation. Nevertheless, in order to discuss projection methods and make the results precise, 
we make the following (standard) definition of a Galerkin approximation.

\begin{definition}[The Galerkin approximation]\label{def:Gal_approx}
Let $\Kspace_k\subseteq\RR^n$ be an $n_k\leq n$ dimensional subspace for $k=0,1,\dots$, and let $\Vbase_k\in\RR^{n\times n_k}$ be a matrix containing an orthogonal basis of $\Kspace_k$.
We call $\hat X_k$ the \emph{Galerkin approximation} to \eqref{eq:GenLyap}, in $\Kspace_k$, if $\hat X_k = \Vbase_k Y_k \Vbase_k^T$ and $Y_k$ is determined by the condition
\begin{align*}
\Vbase_k^T\left(\Lop(\hat X_k) + \Pi(\hat X_k)  + BB^T\right)\Vbase_k = 0,
\end{align*}
known as the \emph{projected problem}.
\end{definition}
It is possible that the projected problem does not have a solution, or that the solution is not unique, and in such case the definition of (the) Galerkin approximation in Definition~\ref{def:Gal_approx} is nonsensical. For the generalized Lyapunov equation there are certain sufficient conditions for the Galerkin approximation to exist and be unique, e.g., the criteria in \cite[Theorem~3.9]{DammBenner}\cite[Theorem~4.1]{DammDirectADI} or \cite[Proposition~3.2]{jarlebring2017krylov}.
In fact, for the symmetric generalized Lyapunov equation with our assumptions on the spectral radius, it is possible to assure the existence of the Galerkin approximation. Since the property \cite[Theorem~4.1~(d)]{DammDirectADI} is preserved by orthogonal projections in the symmetric case, cf. Proposition~\ref{prop:aux_pos_def}. 
However, we will not delve deeper into this topic, and further on simply assume that the projected problem has a unique solution.
Related to the (Galerkin) approximation we also have the (Galerkin) residual.
\begin{definition}[The Galerkin residual]\label{def:Gal_residual}
Let $\Kspace_k$ and $\Vbase_k$ be as in Definition~\ref{def:Gal_approx} and, assuming that it exists, let $\hat X_k$ be the Galerkin approximation in $\Kspace_k$ to \eqref{eq:GenLyap}. We define the \emph{Galerkin residual} as 
\begin{align*}
\R_k := \Lop(\hat X_k) + \Pi(\hat X_k) + BB^T.
\end{align*}
\end{definition}
By using this definition the projected problem is commonly expressed as $\Vbase_k^T\R_k\Vbase_k=0$, which is also known as the \emph{Galerkin orthogonality} condition and is a characteristic feature of the Galerkin residual.

\begin{remark}[Generic residual]\label{rem:gen_residual}
It is possible to define a generic residual, which we with a slight abuse of notation also call $\R_k$. We define it as: for any matrix $\hat X_k \in \RR^{n\times n}$, $\R_k := \Lop(\hat X_k) + \Pi(\hat X_k) + BB^T$. Some of the results and arguments presented are valid for a (generic) residual and others, more specialized, only for the Galerkin residual, or another residual specific to some approximation scheme. However, we will always connect the residual to a specific approximation and the Galerkin residual will always be referenced as such.
\end{remark}

We end this section with a specialization of a classical result from linear algebra, sometimes called the \emph{residual equation}.

\begin{proposition}[A matrix equation for the error]\label{prop:residual}
Let $\Lop$ and $\Pi$ be as in \eqref{eq:L}-\eqref{eq:Pi}, and let $X\in\RR^{n\times n}$ be the solution to \eqref{eq:GenLyap}. Moreover, let $\hat X_k \in \RR^{n\times n}$ be a given matrix, let $\R_k$ be the (generic) residual, and define the error $X^\text{e}_k := X - \hat X_k$. Then
\begin{align*}
\Lop(X^\text{e}_k) + \Pi(X^\text{e}_k) + \R_k = 0.
\end{align*}
\end{proposition}
The result in Proposition~\ref{prop:residual} is useful since it may serve as a starting point to derive iterative algorithms for computing approximations to \eqref{eq:GenLyap}. 
Hence, one strategy for computing updates to the current iterate is to compute, or approximate, $X^e_k$, which Proposition~\ref{prop:residual} allows us to do based on the known, or computable, quantities $\Lop$, $\Pi$ and $\R_k$.
The result for the Lyapunov equation was presented already by Smith in \cite{Smith:1968:Matrix}, and for generalized matrix equations cf. \cite[Section~4.2]{DammDirectADI}, and \cite[Algorithm~2]{Kressner:2015:Truncated}.

\subsection{Bilinear systems}
We recall some control theoretic concepts for \emph{bilinear control systems} of the form 
\begin{align}\label{eq:bil_con_sys}
\Sigma \left\{ \begin{aligned}
\dot{x}(t) &= Ax(t) + \sum_{i=1}^m N_i x(t) w_i(t) + Bu(t)\\
y(t) &= C x(t),
\end{aligned}\right.
\end{align}
with $A,N_i \in \RR^{n\times n}, B\in \RR^{n\times r}$ and $C\in \RR^{r\times n}$ and control inputs $u(t)\in\RR^{r}$ and $w(t)\in\RR^{m}$.

\begin{remark}
We highlight that the bilinear system $\Sigma$ in \eqref{eq:bil_con_sys} differs from the usual notation used in the  literature, see e.g.,  \cite{Ahmad2017,Al-Baiyat:1993:ANew,Benner2012,DammBenner,DammDirectADI,Flagg2015,ZHANG2002}. The formulation~\eqref{eq:bil_con_sys} is convenient since it allows for $m\neq r$. However, the system $\Sigma$ can be put on the usual form by considering the input vector $[w(t)^T u(t)^T]^T$, adding $m$ zero-columns to the beginning of $B$, i.e., considering $[0, B]$, and considering the matrices $N_i = 0$ for $i=m+1,m+2,\dots,m+r$. Thus it is reasonable to consider the same Gramians. The system $\Sigma$ can also be compared to systems from applications, e.g., \cite[Equation~(2)]{Mohler:1980:AnOverview}.
\end{remark}
As in  \cite{Al-Baiyat:1993:ANew}, for a MIMO bilinear system \eqref{eq:bil_con_sys}, we define the controllability and observability Gramian as follows.
\begin{definition}[Bilinear Gramians]\label{def:bil_gramians}
Let $A\in\RR^{n\times n}$ be a stable matrix, $N_1,\dots,N_m\in\RR^{n\times n}$, $B\in\RR^{n\times r}$, and $C\in\RR^{r\times n}$, and consider the bilinear system \eqref{eq:bil_con_sys}.
Moreover, let $P_1(t_1) := e^{At_1}B$, $P_j(t_1,\dots,t_j) := e^{At_j}[N_1P_{j-1}, \dots, N_mP_{j-1}]$ for $j=2,3,\dots$, $Q_1(t_1) := Ce^{At_1}$, and $Q_j(t_1,\dots,t_j) := [N_1^TQ_{j-1}^T,\dots,N_m^TQ_{j-1}^T]^T e^{At_j}$ for $j=2,3,\dots$. We define the controllability and observability Gramian respectively as
\begin{align*}
P &:= \sum_{j=1}^\infty \int_0^\infty \dots \int_0^\infty P_j P_j^T dt_1\dots dt_j\\
Q &:= \sum_{j=1}^\infty \int_0^\infty \dots \int_0^\infty Q_j^T Q_j dt_1\dots dt_j.
\end{align*}
\end{definition}
It is possible that the generalized Gramians from Definition~\ref{def:bil_gramians} do not exist; sufficient conditions are given in, e.g., \cite[Theorem~2]{ZHANG2002}. However, if the Gramians exist we also know that they satisfy the following matrix equations 
\begin{equation}\label{eq:bil_gramians}
\begin{aligned}
AP + PA^T + \sum_{i=1}^m N_i P N_i^T +BB^T &= 0\\
A^TQ + QA + \sum_{i=1}^m N_i^T Q N_i +C^T C &= 0.
\end{aligned}
\end{equation}


In relation to the generalized controllability and observability Gramians, one might also define a generalized cross Gramian similar to the SISO case discussed in \cite{Shaker:2015:Control}. 
Consider the generalized Lyapunov equation~\eqref{eq:GenLyap}, and an approximation $\hat X_k$ with related error $X_k^e$ and residual $\R_k$.
One can easily verify that for the auxiliary system 
\begin{align*}
\Sigma^\text{e} = \left\{ \begin{aligned}
\dot{x}(t) &= Ax(t) + \sum_{i=1}^m N_i x(t) w_i(t) + B_{\R_k}u(t)\\
y(t) &= C_{\R_k} x(t),
\end{aligned}\right.
\end{align*}
where $ B_{\R_k} = US^{1/2}$, and $C_{\R_k} = S^{1/2}V^T$, with $\R_k = USV^T$ being a singular value decomposition of $\R_k$, the associated cross Gramian coincides with the error $X_k^e$. In the special case of a symmetric system, we additionally have the following result.
\begin{proposition}\label{prop:sym_error_gramian}
Consider a symmetric generalized Lyapunov equation. Let $\hat X_k$ be an approximation such that the residual $\R_k = \R_k^T \succeq 0$. Then one can choose $B_{\R_k}=C_{\R_k}^T$ and the error $X^\text{e}_k$ is the controllability and observability Gramian of the auxiliary bilinear system.
\end{proposition}

For what follows, let us recall the definition of the $\mathcal{H}_2$-norm for bilinear systems that was introduced by Zhang and Lam in \cite{ZHANG2002}.

\begin{definition}\label{def:bil_h2}
We define the $\H2$-norm for the bilinear system \eqref{eq:bil_con_sys} as
  \begin{align*}
\|\Sigma \|_{\H2}^2:=\trace\left(\sum_{j=1}^{\infty}\int_0^{\infty}
  \dots \int_0^{\infty} \sum_{\ell_1,\cdots,\ell_j=1}^mg_j^{(\ell_1,\dots,
  \ell_j)}(g_j^{(\ell_1,\dots,\ell_j)})^T \mathrm{d}s_1 \cdots \mathrm{d}s_j \right) ,  
  \end{align*}
  with $g_j^{(\ell_1,\dots,\ell_j)}(s_1,\dots,s_k)
  :=Ce^{As_j}N_{\ell_1}e^{As_{j-1}} N_{\ell_2} \cdots e^{As_1}b_{\ell_j}.$
\end{definition} 
It has been shown in \cite[Theorem~6]{ZHANG2002} that if the Gramians $P$ and $Q$ from Definition~\ref{def:bil_gramians} exist, then it holds that
\begin{align*}
  \| \Sigma \|_{\H2}^2 = \trace\left(CPC^T\right) = \trace \left(B^T Q B\right).
\end{align*}

\section{ALS and $\H2$-optimal model reduction for bilinear systems}\label{sect:H2_connection}

In this section, we discuss a low rank approximation method proposed by Kressner and Sirkovi{\'c} in \cite{Kressner:2015:Truncated}. We show that several results can be generalized from the case of standard Lyapunov equations to the more general form \eqref{eq:GenLyap}. Moreover, we show that in the symmetric case the method allows for an interpretation in terms of $\H2$-optimal model reduction for bilinear control systems. With this in mind, we assume that $A=A^T \prec 0$ and $N_i=N_i^T$ for $i=1,\dots,m.$ If additionally we have that $\rho(\Lop^{-1}\Pi)<1,$ the operator $\mathcal{M}(X):=-\Lop(X)-\Pi(X)$ allows us to define a weighted inner product via
\begin{align}\label{eq:M-norm}
\begin{aligned}
  \langle \cdot,\cdot \rangle_{\mathcal{M}}&\colon \RR^{n\times n} \times \RR^{n \times n} \to \RR \\
  \langle X,Y \rangle_{\mathcal{M}} &= \langle X,\mathcal{M}(Y) \rangle = \trace\left(X^T \mathcal{M}(Y) \right),
\end{aligned}
\end{align}
with a corresponding induced $\mathcal{M}$-norm, also known as \emph{energy norm},
\begin{align*}
\|X\|^2 = \langle X,X \rangle_{\mathcal{M}}.
\end{align*}

\subsection{ALS for the generalized Lyapunov equation}\label{sect:ALS}

In \cite{Kressner:2015:Truncated}, it is suggested to construct iterative approximations $\hat{X}_k$ by rank-1 updates that are locally optimal with respect to the $\mathcal{M}$-norm. To be more precise, assume that $X$ is a solution to the symmetric Lyapunov equation \eqref{eq:GenLyap}, i.e.,
\begin{align*}
 AX + XA + \sum_{i=1}^m N_i X N_i + BB^T=0.
\end{align*}
Given an approximation $\hat{X}_k$, we consider the minimization problem 
\begin{align*}
 \min_{v,w\in \mathbb R^n} \| X-\hat{X}_k - v w^T\|_{\mathcal{M}}^2 &= \langle X-\hat{X}_k - v w^T, X-\hat{X}_k - v w^T \rangle_{\mathcal{M}}. 
\end{align*}
Since the minimization involves the constant term $\| X-\hat{X}_k\|^2_{\mathcal{M}},$ it suffices to focus on
\begin{align}\label{eq:J}
 J(v,w):= \langle vw^T , vw^T \rangle_{\mathcal{M}} - 2 \trace\left( wv^T \R_k\right),
\end{align}
where $\R_k$ is the current residual, i.e.,
\begin{align*}
 \R_k:= A \hat{X}_k + \hat{X}_k A + \sum_{i=1}^m N_i \hat{X}_k N_i + BB^T.
\end{align*}
Locally optimal vectors $v_k$ and $w_k$ are then (approximately) determined via an \emph{alternating linear scheme} (ALS). The main step is to fix one of the two vectors, e.g., $v$ and then minimize the strictly convex objective function in order to obtain an update for $w.$ A pseudocode of the algorithm presented in \cite{Kressner:2015:Truncated} is given in Algorithm~\ref{alg:ALS}.
\begin{algorithm} \label{alg:ALS}
\caption{ALS for the generalized Lyapunov equation \eqref{eq:GenLyap} \cite[Algorithm~1]{Kressner:2015:Truncated}}
\SetKwInOut{Input}{input}\SetKwInOut{Output}{output}
\Input{System: $A,N_1,\dots,N_m\in\RR^{n\times n}$ $\R_k \in\RR^{n\times n}$, $\tol$\\Initial guess: $v, w \in \RR^{n}$}
\Output{New approximation vectors: $v^\text{a},w^\text{a}\in\RR^{n}$}
\BlankLine
\While{Change in $(\frac{v^TAv}{\|v\|^2} + \frac{w^TA^Tw}{\|w\|^2})/2$ larger than $\tol$\vspace{0.8ex}}{
\nl Normalize $w = w/\|w\|$\label{alg_step:ALS:norm_w}\\
\nl $\hat A_1 = A + I(w^TAw) + \sum_{i=1}^m N_i(w^TN_iw)$\label{alg_step:ALS:A1}\\
\nl Solve $\hat A_1 v = -\R_k w$\label{alg_step:ALS:solve_v}\\
\nl Normalize $v = v/\|v\|$\label{alg_step:ALS:norm_v}\\
\nl $\hat A_2 = A + I(v^TAv) + \sum_{i=1}^m (v^TN_i v)N_i$\label{alg_step:ALS:A2}\\
\nl Solve $\hat A_2 w = -\R_k^T v$\label{alg_step:ALS:solve_w}\\
}
\nl Normalize such that $\|w\| = \|v\|$\label{alg_step:ALS:norm_balance}\\
\Return $v^\text{a} = v$, $w^\text{a} = w$
\end{algorithm}
The ALS-based approach for computing new subspace extensions can also, in terms of Proposition~\ref{prop:residual}, be seen as searching for an approximation to $X^\text{e}_k$ of the form $v_kw_k^T$ by iterating $(\Lop(v_kw_k^T) + \Pi(v_kw_k^T) + \R_k)w_k = 0$ when determining $v_k$ and $v_k^T(\Lop(v_kw_k^T) + \Pi(v_kw_k^T) + \R_k) = 0$ when determining $w_k$. This is to say that the error is approximated by a rank-1 matrix, and at convergence this would result in the new residual, $\R_{k+1}$, being left-orthogonal to $v_{k}$ and right-orthogonal to $w_{k}$.
In the symmetric case, local minimizers of \eqref{eq:J} are necessarily symmetric positive semidefinite. This yields the following extension of \cite[Lemma~2.3]{Kressner:2015:Truncated}.

\begin{lemma}\label{lem:2.3_gen}
  Consider the symmetric generalized Lyapunov equation \eqref{eq:GenLyap} and assume that $A\prec 0$, $\rho(\Lop^{-1}\Pi)<1$, and $\R_k=\R_k^T\succeq0$. Let $J$ be as in \eqref{eq:J}.
 Then every local minimum $(v_*,w_*)$ to $J$ is such that $v_*w_*^T$ is symmetric positive semidefinite.
\end{lemma}
\begin{proof}
The proof, naturally, follows along the lines of \cite[Lemma 2.3]{Kressner:2015:Truncated}, and hence without loss of generality we assume that $v_*\neq 0$ , $w_*\neq 0$, and $\|v_*\|=\|w_*\|$. Thus $v_*w_*^T$ is positive semidefinite if and only if $v_*=w_*$. The proof is by contradiction and we assume that $v_*\neq w_*$. Then from the strict convexity of $J$ we have that \begin{align*}
 2 J(v_*,w_*) &< J(v_*,v_*) + J(w_*,w_*).
  \end{align*}
  Simplifying the left-hand-side we get
  \begin{align*}
  &2 J(v_*,w_*) = \\
  &\quad
 -2\trace((v_*w_*^T)^T\Lop(v_*w_*^T)) -2\trace((v_*w_*^T)^T\Pi(v_*w_*^T)) - 4\trace((v_*w_*^T)^T\R_k)
 \\
 &=
 -2 v_*^T\Lop(v_*w_*^T)v_* -2 v_*^T\Pi(v_*w_*^T)w_* - 4 v_*^T\R_k w_*,
 \end{align*}
 and similarly the right-hand-side gives
 \begin{align*}
 J(v_*,v_*) + J(w_*,w_*) =
 &- v_*^T\Lop(v_*v_*^T)v_* - v_*^T\Pi(v_*v_*^T)v_* - 2 v_*^T\R_k v_*
 \\
 &- w_*^T\Lop(w_*w_*^T)w_* - w_*^T\Pi(w_*w_*^T)w_* - 2 w_*^T\R_k w_*.
 \end{align*}
 Collecting the terms involving the  $\Lop$-operator we observe that 
 \begin{align*}
-2 v_*^T\Lop(v_*w_*^T)w_* + v_*^T\Lop(v_*v_*^T)v_* + w_*^T\Lop(w_*w_*^T)w_* &=\\
 2(v_*^Tv_*)(w_*^TAw_* - v_*^TAv_*) + 2w_*^Tw_*(v_*^TAv_* - w_*^TAw_*) &= 0,
 \end{align*}
 and thus the inequality reduces to
 \begin{align*}
 &-2 v_*^T\Pi(v_*w_*^T)w_* - 4 v_*^T\R_k w_* <\\&\quad - v_*^T\Pi(v_*v_*^T)v_* - 2 v_*^T\R_k v_* - w_*^T\Pi(w_*w_*^T)w_* - 2 w_*^T\R_k w_*.
 \end{align*}
 We rewrite the inequality  as
  \begin{align*}
 -2 v_*^T\Pi(v_*w_*^T)w_* + v_*^T\Pi(v_*v_*^T)v_* + w_*^T\Pi(w_*w_*^T)w_* 
 < -2 (v_*-w_*)^T \R_k (v_*-w_*).
 \end{align*}
 Now if we can show that 
 \begin{align*}
 -2 v_*^T\Pi(v_*w_*^T)w_* + v_*^T\Pi(v_*v_*^T)v_* + w_*^T\Pi(w_*w_*^T)w_* \geq 0,
 \end{align*}
 this would imply that $-2 (v_*-w_*)^T \R_k (v_*-w_*)>0$ in contradiction to the positive semidefiniteness of $\R_k$.
The argument is hence concluded by considering the expression $-2 v_*^T\Pi(v_*w_*^T)w_* + v_*^T\Pi(v_*v_*^T)v_* + w_*^T\Pi(w_*w_*^T)w_*$. We can without loss of generality consider $m=1$, hence only one $N$-matrix, since the following argument can be applied to all terms in the sum independently. Observe that
 \begin{align*}
 -2 v_*^T N v_*w_*^T N w_* + v_*^TNv_*v_*^TNv_* + w_*^TNw_*w_*^TNw_* = 
 (v_*^TNv_* - w_*^TNw_*)^2 \geq 0,
 \end{align*}
 which shows the desired inequality, and thus concludes the proof.
 \qed
\end{proof}
Algorithm~\ref{alg:ALS} and the argument in Lemma~\ref{lem:2.3_gen} is based on a residual. However, if one considers $\hat X_k = 0$, then $\R_k = BB^T$, and hence the result is applicable directly to any symmetric generalized Lyapunov equation. The focus on the residual in the previous results is natural since it leads to the following extension of \cite[Theorem~2.4]{Kressner:2015:Truncated} to the case of symmetric generalized Lyapunov equations.
\begin{theorem}\label{thm:2.4_gen}
  Consider the symmetric generalized Lyapunov equation \eqref{eq:GenLyap} with the additional assumptions that $A\prec 0$ and $\rho(\Lop^{-1}\Pi)<1$. Moreover, consider the sequence of approximations constructed as
  \begin{align}\label{eq:ALS_iter}
  \begin{aligned}
  \hat X_0 &= 0\\
  \hat X_{k+1} &= \hat X_k + v_{k+1}v_{k+1}^T, \qquad k = 0,1,\dots,
 \end{aligned}
\end{align} 
where $v_{k+1}$ is a locally optimal vector computed with ALS (Algorithm~\ref{alg:ALS}).
 Then $\R_k=\R_k^T \succeq 0$ for all $k\ge 0$.
\end{theorem}
\begin{proof}
We show the assertion by induction. It trivially holds that $\R_0=\R_0^T\succeq 0$. Now assume that this is the case for some $\R_k$. 
From Lemma~\ref{lem:2.3_gen} the local minimizers of \eqref{eq:J} are symmetric and hence $\hat X_{k+1}$ as in \eqref{eq:ALS_iter} is reasonable as defined. Moreover, since $\hat X_{k+1}$ and the operators in \eqref{eq:GenLyap} are symmetric we know that $\R_{k+1}$ is symmetric. Thus what is left to show is that $\R_{k+1} \succeq0$.

We have that $\R_{k+1} \succeq0$ if and only if $y^T\R_{k+1}y\geq0$ for all $y\in\RR^{n}$. Hence take an arbitrary $y\in\RR^{n}$ and consider $y^T\R_{k+1}y$. We derive properties similar to \cite[equations~(12)-(14)]{Kressner:2015:Truncated}: 

Since $(v_{k+1},v_{k+1})$ is a (local) minimizer of $J(v,w)$, it also follows that $v_{k+1}$ is a (global) minimizer of the (convex) cost function
$$J_w(v):=J(v,w)=\langle vw^T,vw^T\rangle _{\mathcal{M}} - 2\trace(wv^T\R_k),$$
where $w=v_{k+1}$. Note that the gradient $\nabla_vJ_w$ of $J_w$ with respect to $v$ is given by
\begin{align*}
 (\nabla_vJ_w)_i =2\langle e_iw^T,vw^T \rangle_{\mathcal{M}}- 2e_i^T \R_k w .
\end{align*}
Due to the optimality of $v_{k+1}$ with respect to $J_{v_{k+1}}$, first order optimality conditions then imply that
\begin{align}\label{eq:13_gen_b}
 -Av_{k+1}v_{k+1}^Tv_{k+1}-v_{k+1}v_{k+1}^TAv_{k+1}-\sum_{i=1}^mN_iv_{k+1}v_{k+1}^TN_iv_{k+1}=\R_k v_{k+1} .
\end{align}
Striking this equality with $v_{k+1}^T$  from the left implies that
\begin{align}\label{eq:14_gen_b}
2  v_{k+1}^TA v_{k+1} \|v_{k+1}\|^2 = - v_{k+1}^T \R_{k} v_{k+1} - \sum_{i=1}^m ( v_{k+1}^T N_i   v_{k+1})^2 .
\end{align}
Based on this we can write the residual as
{
\begin{align*}
&y^T\R_{k+1}y = y^T\R_{k} y + y^T\left(\Lop(v_{k+1}v_{k+1}^T) + \Pi(v_{k+1}v_{k+1}^T)\right)y
\\&=
 y^T\R_{k} y + y^T\left(Av_{k+1}v_{k+1}^T + v_{k+1}v_{k+1}^TA + \sum_{i=1}^m N_iv_{k+1}v_{k+1}^T N_i\right)y
\\&=
 y^T\R_{k} y + \sum_{i=1}^m y^TN_iv_{k+1}v_{k+1}^T N_i y + \frac{1}{\|v_{k+1}\|^2}\Bigg(
\\&\quad +  y^T\left(-\R_{k} v_{k+1}v_{k+1}^T - ( v_{k+1}^TA v_{k+1})  v_{k+1}v_{k+1}^T - \sum_{i=1}^m ( v_{k+1}^T N_i v_{k+1})N_iv_{k+1}v_{k+1}^T\right)y
\\&\quad +  y^T\left(-v_{k+1}  v_{k+1}^T\R_{k} - v_{k+1}v_{k+1}^T ( v_{k+1}^TAv_{k+1})  - \sum_{i=1}^m ( v_{k+1}^T N_iv_{k+1})v_{k+1}v_{k+1}^T N_i\right)y\Bigg).
\end{align*}}%
Here, the third equality follows from \eqref{eq:13_gen_b} and its transpose by exploiting the symmetry of the involved matrices. We rearrange, identify the term $-2( v_{k+1}^T A  v_{k+1})  v_{k+1} v_{k+1}^T$ and insert \eqref{eq:14_gen_b} to get
{
\begin{align*}
&y^T\R_{k+1}y
=
 y^T\R_{k} y
\\&\quad + \frac{1}{\|v_{k+1}\|^2}y^T\left(-\R_{k} v_{k+1}v_{k+1}^T -v_{k+1}v_{k+1}^T\R_{k} + \frac{1}{\|v_{k+1}\|^2}v_{k+1}^T\R_{k}v_{k+1} v_{k+1} v_{k+1}^T \right)y
\\&\quad +\frac{1}{\|v_{k+1}\|^2} y^T \left(\sum_{i=1}^m N_i v_{k+1}v_{k+1}^TN_i \|v_{k+1}\|^2 + \frac{1}{\|v_{k+1}\|^2}\sum_{i=1}^m ( v_{k+1}^TN_i v_{k+1})^2 v_{k+1}v_{k+1}^T\right) y
\\&\quad +\frac{1}{\|v_{k+1}\|^2} y^T\left( - \sum_{i=1}^m ( v_{k+1}^T N_i v_{k+1})N_iv_{k+1}v_{k+1}^T - \sum_{i=1}^m ( v_{k+1}^T N_iv_{k+1})v_{k+1}v_{k+1}^T N_i\right)y \displaybreak[0]
\\[2ex]&=
 y^T\R_{k} y 
\\&\quad + \frac{1}{\|v_{k+1}\|^2}y^T\left(-\R_{k} v_{k+1}v_{k+1}^T -v_{k+1}v_{k+1}^T\R_{k} + \frac{1}{\|v_{k+1}\|^2}v_{k+1}^T\R_{k}v_{k+1} v_{k+1} v_{k+1}^T \right)y
\\&\quad +  \frac{1}{\|v_{k+1}\|^2}\Bigg(\sum_{i=1}^m (y^TN_iv_{k+1})^2\|v_{k+1}\|^2 +  \frac{1}{\|v_{k+1}\|^2} ( v_{k+1}^TN_i v_{k+1})^2 ( v_{k+1}^Ty)^2 \\ 
&\hspace{4cm} - 2 ( v_{k+1}^T N_i v_{k+1}) (y^TN_i v_{k+1})( v_{k+1}^T y ) \Bigg) \displaybreak[0]
\\[2ex]&=
y^T\R_{k}y 
\\&\quad+ \frac{1}{\|v_{k+1}\|^2} \left( -2(y^T\R_{k} v_{k+1})(v_{k+1}^Ty) + \frac{1}{\|v_{k+1}\|^2}( v_{k+1}^T\R_{k}v_{k+1})( v_{k+1}^Ty)^2  \right)
\\&\quad + \frac{1}{\|v_{k+1}\|^2}\sum_{i=1}^m\left( \|v_{k+1}\|(y^TN_i v_{k+1}) -  \frac{1}{\|v_{k+1}\|}( v_{k+1}^TN_i v_{k+1})( v_{k+1}^T y)\right)^2 \displaybreak[0]
\\[2ex]&=
(y - v_{k+1} \frac{v_{k+1}^Ty}{\|v_{k+1}\|^2} )^T\R_{k}(y - v_{k+1} \frac{v_{k+1}^Ty}{\|v_{k+1}\|^2} )
\\&\quad + \frac{1}{\|v_{k+1}\|^2}\sum_{i=1}^m\left( \|v_{k+1}\|(y^TN_i v_{k+1}) -  \frac{1}{\|v_{k+1}\|}( v_{k+1}^TN_i v_{k+1})( v_{k+1}^T y)\right)^2
\\ &\geq 0.
\end{align*}}%
This asserts the inductive step and hence concludes the proof.
\qed
\end{proof}

\begin{corollary}
The iteration \eqref{eq:ALS_iter} produces an increasing sequence of approximations $0=\hat X_0 \preceq \hat X_1\preceq\dots\preceq X$.
\end{corollary}

\subsection{$\H2$-optimal model reduction for symmetric state space systems}

For the standard Lyapunov equation it has been shown, in \cite{Benner:2014:optimality}, that minimization of the energy norm induced by the Lyapunov operator (see \cite{vandereycken2010riemannian}) is related to $\H2$-optimal model reduction for linear control systems. As it turns out, a similar conclusion can be drawn for the minimization of the cost functional \eqref{eq:J} and $\H2$-optimal model reduction for symmetric bilinear control systems. In this regard, let us briefly summarize the most important concepts from bilinear model reduction. Given a bilinear system $\Sigma$ as in \eqref{eq:bil_con_sys} with $\mathrm{dim}(\Sigma)=n,$ the goal of model reduction is to construct a surrogate model $\widehat{\Sigma}$ of the form
\begin{align*}
\widehat{\Sigma}\colon \left\{ \begin{aligned}
\dot{\widehat{x}}(t) &= \widehat{A}\widehat{x}(t) + \sum_{i=1}^m \widehat{N}_i \widehat{x}(t) w_i(t) + \widehat{B}u(t)\\
\widehat{y}(t) &= \widehat{C} \widehat{x}(t),
\end{aligned}\right.
\end{align*}
with $\widehat{A},\widehat{N}_i \in \RR^{k\times k}, \widehat{B}\in \RR^{n\times k}, C\in \RR^{r\times k}$ and control inputs $u(t)\in\RR^{r}$ and $w(t)\in\RR^{m}$. In particular, the reduced system should satisfy $k\ll n$ and $\widehat{y}(t)\approx y(t)$ in some norm. For the bilinear $\H2$-norm defined in Definition \ref{def:bil_h2}, in  \cite{Benner2012,Flagg2015}, the authors have suggested an algorithm, BIRKA, that iteratively tries to compute a reduced model satisfying first order necessary conditions for $\H2$-optimality. A corresponding pseudocode is given in Algorithm~\ref{alg:BIRKA}.

\begin{algorithm} \label{alg:BIRKA}
\caption{BIRKA \cite[Algorithm~2]{Benner2012} and \cite[Algorithm~5]{Flagg2015}}
\SetKwInOut{Input}{input}\SetKwInOut{Output}{output}
\Input{System: $A,N_1,\dots,N_m\in\RR^{n\times n}$ $B\in\RR^{n\times r}$, $C\in\RR^{r\times n}$, $\tol$\\Initial guess: $\tilde V, \tilde W \in \RR^{n\times k}$}
\Output{Optimal $\H2$-approximation spaces: $\tilde V^\text{opt},\tilde W^\text{opt} \in \RR^{n\times k}$}
\BlankLine
\While{Change in eigenvalues of $\tilde A$ larger than $\tol$}{
\nl Update guess $\tilde A = (\tilde W^T \tilde V)^{-1}\tilde W^T A \tilde V$, $\tilde N_1 = (\tilde W^T \tilde V)^{-1}\tilde W^T N_1\tilde V$, \dots, $\tilde N_m = (\tilde W^T \tilde V)^{-1}\tilde W^T N_m\tilde V$, $\tilde B = (\tilde W^T \tilde V)^{-1}\tilde W^T B$, $\tilde C = C\tilde V$\label{alg_step:BIRKA:update}\\
\nl Decompose $R\tilde \Lambda R^{-1} = \tilde A$\label{alg_step:BIRKA:eig}\\
\nl Compute $\hat B = R^{-T}\tilde B$, $\hat C = \tilde C R$, $\hat N_1 = R^{-1} \tilde N_1 R$, \dots, $\hat N_m = R^{-1} \tilde N_m R$\label{alg_step:BIRKA:transform_diag}\\
\nl Solve $\tilde V \tilde \Lambda + A\tilde V + \sum_{i=1}^m N_i \tilde V \hat N_i^T + B\hat B^T = 0$ for $\tilde V$ \label{alg_step:BIRKA:solve_V}\\
\nl Solve $\tilde W \tilde \Lambda + A^T\tilde W + \sum_{i=1}^m N_i^T \tilde W \hat N_i + C^T\hat C = 0$ for $\tilde W$ \label{alg_step:BIRKA:solve_W}\\
\nl Orthogonalize $\tilde V = \orth(\tilde V)$, $\tilde W = \orth(\tilde W)$\label{alg_step:BIRKA:orth}\\
}
\Return $\tilde V^\text{opt} = \tilde V$, $\tilde W^\text{opt} = \tilde W$ 
\end{algorithm}

To establish the connection we introduce the following generalizations of the operator $\mathcal{M}$:
\begin{align*}  
  \widetilde{\mathcal{M}}&\colon \mathbb R^{n\times r} \to \mathbb R^{n\times r}, \quad \widetilde{\mathcal{M}}(X) := - A X -  X  \hat{A} -\sum_{i=1}^mN_i X   \hat{N}_i, \\
  \widehat{\mathcal{M}}&\colon \mathbb R^{r\times r} \to \mathbb R^{r\times r}, \quad  \widehat{\mathcal{M}}(X) := - \hat{A}  X -  X  \hat{A} -\sum_{i=1}^m \hat{N}_i X \hat{N}_i,
\end{align*}
where $\hat{A}=V^T A V, \hat{N}_i=V^TN_i V$ for $i=1,\dots,m$, and $V\in\RR^{n\times r}$ is orthogonal. Our first result is concerned with the invertibility of the operators $\widetilde{\mathcal{M}}$ and $\widehat{\mathcal{M}}$, respectively. 

\begin{proposition}\label{prop:aux_pos_def}
 If $\sigma(\mathcal{M})=-\sigma(\Lop+\Pi)\subset \mathbb C_+$ then  $\sigma(\widetilde{\mathcal{M}}) \subset \mathbb C_+$ and $\sigma(\widehat{\mathcal{M}}) \subset \mathbb C_+$ . 
\end{proposition}
\begin{proof}
  
  Note that $\sigma(\widetilde{\mathcal{M}})$ is determined by the eigenvalues of the matrix 
  \begin{align}\label{eq:aux_Mtilde}
   \widetilde{\mathbf{M}} := -I\otimes A  - \hat{A} \otimes I - \sum_{i=1}^m \hat{N}_i \otimes  N_i .
  \end{align}
  Similarly, we obtain  $\sigma(\mathcal{M})$ by computing the eigenvalues of the matrix 
  \begin{align}\label{eq:aux_Mhat}
   \mathbf{M} := -I\otimes A - A   \otimes I - \sum_{i=1}^m  N_i\otimes  N_i .
  \end{align}
Since $A$ and $N_i$ are assumed to be symmetric, we conclude that $\mathbf{M}=\mathbf{M}^T\succ 0$. Let us then define the orthogonal matrix $\mathbf{V}=  I\otimes V$. It follows that $\widetilde{\mathbf{M}}=\mathbf{V}^T \mathbf{M} \mathbf{V}$ and, consequently, $\widetilde{\mathbf{M}} =\widetilde{\mathbf{M}}^T\succ 0$. A similar argument with $\mathbf{V}=V\otimes V$ can be applied to show the second assertion.
\qed
\end{proof}


Given a reduced bilinear system, we naturally obtain an approximate solution to the generalized Lyapunov equation. Moreover, the error with respect to the $\mathcal{M}$ inner product is given by the $\H2$-norms of the original and reduced system, respectively.
\begin{proposition}\label{prop:ip_h2}
  Let $\Sigma$ denote a bilinear system \eqref{eq:bil_con_sys} with $A=A^T\prec 0,N_i=N_i^T, i=1,\dots,m$ and $B=C^T$ be given. Assume that $\rho(\Lop^{-1}\Pi)<1.$ Given an orthogonal  $V\in \RR^{n \times k},k< n,$ define a reduced bilinear system $\hat{\Sigma}$ via $\hat{A}=V^T A V, \hat{N}_i=V^TN_i V$ and $\hat{B}=V^TB=\hat{C}^T.$ Let $X$ be the solution to \eqref{eq:GenLyap}, and let $\hat{X}$ be the solution to $\hat{A}\hat{X}+\hat{X}\hat{A}+\sum_{i=1}^m \hat{N}_i \hat{X} \hat{N}_i +\hat{B}\hat{B}^T=0.$ Then 
  \begin{align*}
    \| X-V\hat{X}V^T \|_{\mathcal{M}}^2 = \| \Sigma \|_{\H2}^2  - \| \hat{\Sigma} \|_{\H2}^2.
  \end{align*}
\end{proposition}
\begin{proof}
  By assumption it holds that the controllability Gramian $X$ exists and the operator $\mathcal{M}$ is invertible. Moreover, for the reduced system we obtain
  \begin{align*}
\widehat{\mathcal{M}}(\hat{X}) &= -\hat{A} \hat{X} - \hat{X}\hat{A} - \sum_{i=1}^m \hat{N}_i \hat{X} \hat{N}_i 
= -V^T ( AV\hat{X}V^T + V\hat{X}V^T A + \sum_{i=1}^m N_i V\hat{X}V^T N_i ) V \\
&= V^T \mathcal{M}(V\hat{X}V^T)V .
\end{align*}
Hence, we obtain 
\begin{align*}
   \langle X-V\hat{X}V^T , X -& V\hat{X} V^T \rangle_{\mathcal{M}} 
   = \langle X-V\hat{X}V^T, \mathcal{M}(X)-\mathcal{M}(V\hat{X}V^T) \rangle \\
  & = \langle X-V\hat{X}V^T, \mathcal{M}(X) \rangle - \langle X-V\hat{X}V^T, \mathcal{M}(V\hat{X}V^T) \rangle \\ 
  & = \langle X-V\hat{X}V^T, BB^T \rangle - \langle \mathcal{M}(X),V\hat{X}V^T \rangle + \langle V\hat{X}V^T, \mathcal{M}(V\hat{X}V^T)\rangle \\ 
  & = \trace(B^TXB) -\trace(B^TV\hat{X}V^TB) \\&\quad- \trace(B^T V\hat{X}V^TB)+\trace(\hat{X}V^T \mathcal{M}(V\hat{X}V^T)V) \\
  & = \| \Sigma\| _{\H2}^2 - 2 \| \hat{\Sigma}\|_{\H2}^2 + \trace(\hat{X} \widehat{\mathcal{M}}(\hat{X}) ) 
   = \| \Sigma\| _{\H2}^2 -  \| \hat{\Sigma}\|_{\H2}^2.
\end{align*}
\qed
\end{proof}

Extending the results from \cite{Benner:2014:optimality}, we obtain a lower bound for the previous terms by means of the $\H2$-norm of the error system $\Sigma-\widehat{\Sigma}$.

\begin{proposition}\label{prop:err_sys_low_bound}
  Let $\Sigma$ denote a bilinear system \eqref{eq:bil_con_sys} with $A=A^T\prec 0,N_i=N_i^T, i=1,\dots,m$ and $B=C^T$ be given. Assume that $\rho(\Lop^{-1}\Pi)<1$. Given an orthogonal  $V\in \RR^{n \times k},k< n,$ define a reduced bilinear system $\widehat{\Sigma}$ via $\hat{A}=V^T A V, \hat{N}_i=V^TN_i V$ and $\hat{B}=V^TB=\hat{C}^T.$ Then, it holds
  \begin{align*}
    \| \Sigma- \hat{\Sigma} \|_{\H2}^2 \le \| \Sigma \|_{\H2}^2  - \| \hat{\Sigma} \|_{\H2}^2,
  \end{align*}
  with equality if $\widehat{\Sigma}$ is locally $\H2$-optimal.
\end{proposition}

\begin{proof}
The proof follows by arguments similar to those used in \cite[Lemma 3.1]{Benner:2014:optimality}.
  By definition of the $\H2$-norm for bilinear systems we have 
  \begin{align*}
    \| \Sigma- \widehat{\Sigma}\|_{\H2}^2 = \trace(\begin{bmatrix} B^T & -\hat{B}^T \end{bmatrix} X_e \begin{bmatrix} B \\ -\hat{B} \end{bmatrix}) ,
  \end{align*}
where $X_e= \begin{bmatrix} X & Y \\ Y^T & \hat{X}\end{bmatrix}$ is the solution of 
\begin{align*}
  \begin{bmatrix} A & 0\\ 0 & \hat{A} \end{bmatrix} X_e + X_e \begin{bmatrix} A & 0\\ 0 & \hat{A} \end{bmatrix} + \sum_{i=1}^m \begin{bmatrix} N_i & 0\\ 0 & \hat{N}_i \end{bmatrix} X_e \begin{bmatrix} N_i & 0\\ 0 & \hat{N}_i \end{bmatrix} + \begin{bmatrix} B \\ \hat{B} \end{bmatrix}\begin{bmatrix} B^T & \hat{B}^T \end{bmatrix}=0.
\end{align*}
Analyzing the block structure of $X_e$ we thus find the equivalent expression 
\begin{align*}
\| \Sigma- \hat{\Sigma}\|_{\H2}^2 &= \trace(B^TXB-\hat{B}^T \hat{X}\hat{B} - 2 \begin{bmatrix} B^T,-\hat{B}^T\end{bmatrix} \begin{bmatrix}Y \\ \hat{X}\end{bmatrix} \hat{B} )\\
&= \|\Sigma\|_{\H2}^2 - \| \hat{\Sigma}\|^2_{\H2}-2\trace( \begin{bmatrix} B^T,-\hat{B}^T\end{bmatrix} \begin{bmatrix}Y \\ \hat{X}\end{bmatrix} \hat{B} ).
\end{align*}
We claim that $\trace( \begin{bmatrix} B^T,-\hat{B}^T\end{bmatrix} \begin{bmatrix}Y \\ \hat{X}\end{bmatrix} \hat{B} )\ge 0$ which then shows the first assertion. In fact, $Y$ and $\hat{X}$ are the solutions of
\begin{align*}
 AY + Y\hat{A} + \sum_{i=1}^m N_i Y \hat{N}_i + B\hat{B}^T&=0 \\
 \hat{A}\hat{X} + \hat{X}\hat{A} + \sum_{i=1}^m \hat{N}_i \hat{X} \hat{N}_i + \hat{B}\hat{B}^T&=0.
\end{align*}
With the operators introduced in \eqref{eq:aux_Mtilde} and \eqref{eq:aux_Mhat}, we obtain 
\begin{align*}
 \trace(B^T Y \hat{B}) -& \trace(\hat{B}^T \hat{X}\hat{B}) = \mathrm{vec}(B\hat{B}^T)^T \mathrm{vec}(Y) -  \mathrm{vec}(\hat{B}\hat{B}^T)^T \mathrm{vec}(\hat{X}) \\
 &=  \mathrm{vec}(B\hat{B}^T)^T \widetilde{\mathbf{M}}^{-1}\mathrm{vec}(B\hat{B}^T) -  \mathrm{vec}(\hat{B}\hat{B}^T)^T \widehat{\mathbf{M}}^{-1}\mathrm{vec}(\hat{B}\hat{B}^T) \\
 &= \mathrm{vec}(B\hat{B}^T)^T\left( \widetilde{\mathbf{M}}^{-1} - \mathbf{V} ( \mathbf{V}^T \widetilde{\mathbf{M}} \mathbf{V})^{-1} \mathbf{V} \right)\mathrm{vec}(B\hat{B}^T).
\end{align*}
As in \cite[Lemma 3.1]{Benner:2014:optimality}, it follows that the previous expression contains the Schur complement of $\widetilde{\mathbf{M}}^{-1}$ in $\mathbf{S}= \begin{bmatrix}\mathbf{V}^T \widetilde{\mathbf{M}} \mathbf{V} & \mathbf{V}^T \\ \mathbf{V} & \widetilde{\mathbf{M}}^{-1} \end{bmatrix} $ which can be shown to be positive semidefinite. We omit the details here and instead refer to \cite{Benner:2014:optimality}. 

Assume now that $\widehat{\Sigma}$ is locally $\H2$-optimal. From \cite{ZHANG2002}, we have the following first-order necessary optimality conditions
\begin{align*}
  Y^T Z + \hat{X} \hat{Z} &= 0, \quad 
  Z^T N_iY + \hat{X}\hat{N}_i\hat{Z}=0, \ \ i=1,\dots,m, \\ 
  Z^TB + \hat{Z}\hat{B} &=0, \quad  CY+\hat{C}\hat{X} =0,
\end{align*}
where $Y,\hat{X}$ are as before and $Z,\hat{Z}$ satisfy 
\begin{align*}
 A^T Z + Z\hat{A} + \sum_{i=1}^m N_i^T Z \hat{N}_i -C^T \hat{C } &=0, \\
 \hat{A}^T \hat{Z} + \hat{Z}\hat{A} + \sum_{i=1}^m \hat{N}^T_i \hat{Z} \hat{N}_i + \hat{C}^T \hat{C} &=0. 
\end{align*}
From the symmetry of $A,\hat{A},N_i$ and $\hat{N}_i$ as well as the fact that $B=C^T$ and $\hat{B}=\hat{C}^T$, we conclude that $\hat{Z}=\hat{X}$ and $Z=-Y$. Hence, from the optimality conditions, we obtain 
\begin{align*}
 0=Z^TB+\hat{Z}\hat{B}=-Y^TB+\hat{X}\hat{B} 
\end{align*}
which in particular implies that $\trace( \begin{bmatrix} B^T,-\hat{B}^T\end{bmatrix} \begin{bmatrix}Y \\ \hat{X}\end{bmatrix} \hat{B} )= 0.$ This shows the second assertion.
\qed
\end{proof}
As a consequence from Proposition \ref{prop:ip_h2} and Proposition \ref{prop:err_sys_low_bound}, we obtain the following result.
 
\begin{theorem}
Let $\Sigma$ denote a bilinear system \eqref{eq:bil_con_sys} with $A=A^T\prec 0,N_i=N_i^T, i=1,\dots,m$ and $B=C^T$ be given. Assume that $\rho(\Lop^{-1}\Pi)<1$. Given an orthogonal  $V\in \RR^{n \times k},k< n,$ define a reduced bilinear system $\widehat{\Sigma}$ via $\hat{A}=V^T A V, \hat{N}_i=V^TN_i V$ and $\hat{B}=V^TB=\hat{C}^T.$ Assume that $\hat{X}$ solves $\hat{A}\hat{X}+\hat{X}\hat{A}+\sum_{i=1}^m \hat{N}_i \hat{X} \hat{N}_i^T +\hat{B}\hat{B}^T=0.$ If $\widehat{\Sigma}$ is locally $\H2$-optimal, then $V\hat{X}V^T$ is locally optimal with respect to the $\mathcal{M}$-norm.
\end{theorem}

\subsection{Equivalence of ALS and rank-1 BIRKA}
So far we have shown that having a subspace producing a locally $\H2$-optimal model reduction implies having a subspace for which the Galerkin approximation is locally optimal in the $\mathcal{M}$-norm.
In this part we, algorithmically, establish an equivalence between BIRKA and ALS. More precisely, for the symmetric case the equivalence is between BIRKA applied with the target model reduction subspace of dimension 1 for \eqref{eq:bil_con_sys}, and ALS applied to \eqref{eq:GenLyap}. The proof is based on the following lemmas.

\begin{lemma}\label{lem:rank1_BIRKA}
Consider using BIRKA (Algorithm~\ref{alg:BIRKA}) with $k=1$, i.e., having initial guesses and output being vectors. Then $\tilde A\in\RR$ is a scalar and hence we can take $\tilde \Lambda = \tilde A$ and $R=1$ in Step~\ref{alg_step:BIRKA:eig}. Thus $\hat B = \tilde B$, $\hat C = \tilde C$, $\hat N_1 = \tilde N_1$, \dots, $\hat N_m = \tilde N_m$, and hence Steps~\ref{alg_step:BIRKA:eig}-\ref{alg_step:BIRKA:transform_diag} are redundant. Moreover, since $\tilde V$ and $\tilde W$ are vectors, Step~\ref{alg_step:BIRKA:orth}, is redundant.
\end{lemma}
\begin{proof}
By direct computation.
\qed
\end{proof}

When speaking about \emph{redundant} steps and operations we mean that the entities assigned in that step is exactly equal to another, existing, entity. In such a situation the algorithm can be rewritten, by simply changing the notation, in a way that skips the redundant step and still produces the same result.

\begin{lemma}\label{lem:sym_BIRKA}
Consider a symmetric generalized Lyapunov equation and let $v,w\in\RR^{n}$ be two given vectors.
Let $v_\textsc{birka},w_\textsc{birka}\in\RR^{n}$ be the approximations obtained by applying BIRKA (Algorithm~\ref{alg:BIRKA}) to \eqref{eq:GenLyap} with $C = B^T$ and initial guesses $v$ and $w$. If $v=w$, then $v_\textsc{birka} = w_\textsc{birka}$.
\end{lemma}
\begin{proof}
The proof is by induction, and it suffices to show that if $\tilde V = \tilde W$ at the beginning of a loop, the same holds at the end of the loop. Thus assume $\tilde V = \tilde W$. Then $\tilde N_i = (\tilde W^T \tilde V)^{-1} \tilde W^T N_i \tilde V = \tilde V^T N_i \tilde V/\|V\|^2 = \tilde V^T N_i^T \tilde V/\|V\|^2 = \hat N_i^T$ for $i=1,\dots,m$, and $\tilde C = C \tilde V = B^T \tilde W = \tilde B^T$. By Lemma~\ref{lem:rank1_BIRKA} we do not need to consider Steps~\ref{alg_step:BIRKA:eig}-\ref{alg_step:BIRKA:transform_diag}. We can now conclude that Step~\ref{alg_step:BIRKA:solve_V} and Step~\ref{alg_step:BIRKA:solve_W} are equal, and thus at the end of the iteration we still have $\tilde V = \tilde W$.
\qed
\end{proof}

\begin{lemma}\label{lem:sym_ALS}
Consider a symmetric generalized Lyapunov equation and let $v,w\in\RR^{n}$ be two given vectors.
Let $v_\textsc{als},w_\textsc{als}\in\RR^{n}$  be the approximations obtained by applying the ALS algorithm (Algorithm~\ref{alg:ALS}) to \eqref{eq:GenLyap} with initial guesses $v$ and $w$.
If $v=w$, then $v_\textsc{als} = w_\textsc{als}$.
\end{lemma}
\begin{proof}
Similar to the proof of Lemma~\ref{lem:sym_BIRKA} it is enough to show that if $v = w$ at the beginning of a loop then it also holds at the end of the loop. Hence we assume that $v = w$. Then $\hat A_1 = \hat A_2$ follows by direct calculations. 
Moreover, by assumption $\R_k=\R_k^T$. Thus Step~\ref{alg_step:ALS:solve_v} and Step~\ref{alg_step:ALS:solve_w} are equal, and hence at the end of the iteration we still have that $v=w$.
\qed
\end{proof}

\begin{theorem}\label{thm:H2-iter-alg}
Consider a symmetric generalized Lyapunov equation and let $v\in\RR^{n}$ be a given vector. Let $v_\textsc{birka}\in\RR^{n}$ be the approximation obtained by applying BIRKA (Algorithm~\ref{alg:BIRKA}) to \eqref{eq:GenLyap} with $C= B^T$ and initial guess $v$. Moreover, let $v_\textsc{als}\in\RR^n$  be the approximation obtained by applying the ALS algorithm (Algorithm~\ref{alg:ALS}) to \eqref{eq:GenLyap} with initial guess $v$.
Then $v_\textsc{birka}= v_\textsc{als}$.
\end{theorem}
\begin{proof}
Firstly, Lemma~\ref{lem:sym_BIRKA} and Lemma~\ref{lem:sym_ALS} makes it reasonable to assess the algorithms with only a single initial guess as well as a single output. Moreover, Step~\ref{alg_step:BIRKA:solve_W} in BIRKA as well as Steps~\ref{alg_step:ALS:A1}-\ref{alg_step:ALS:norm_v} in ALS are redundant. Furthermore, we know from Lemma~\ref{lem:rank1_BIRKA} that in this situation Steps~\ref{alg_step:BIRKA:eig}, \ref{alg_step:BIRKA:transform_diag}, and \ref{alg_step:BIRKA:orth} of BIRKA are also redundant. Hence we need to compare the procedure consisting of Steps~\ref{alg_step:BIRKA:update} and \ref{alg_step:BIRKA:solve_V} from BIRKA, with the procedure consisting of Steps~\ref{alg_step:ALS:norm_w}, \ref{alg_step:ALS:A2}, and \ref{alg_step:ALS:solve_w} from ALS. It can be observed that the computations are equivalent and thus the asserted equality holds if they stop after an equal amount of iterations. We hence consider the stopping criteria and note that they are the same, since 
$(v^TA^Tv + v^TAv)/2\|v\|^2 = v^TAv/\|v\|^2 = \tilde A \in \RR$.
\qed
\end{proof}

\begin{corollary}
Theorem~\ref{thm:2.4_gen} is applicable with ALS changed to BIRKA, using subspaces of dimension 1.
\end{corollary}

\section{Fixed-point iteration and approximative $\mathcal{M}$-norm minimization}\label{sect:fixed-point}
In this section we present a similar result as in the previous section, but for the fixed-point iteration. Instead of (locally) minimizing the error in the $\mathcal{M}$-norm with rank-1 updates, the fixed-point iteration minimizes an upper bound for the $\mathcal{M}$-norm, but with no rank-constraint on the minimizer.
As in the previous section, we assume that $A=A^T \prec 0$, and $N_i=N_i^T$ for $i=1,\dots,m$, as well as $\rho(\Lop^{-1}\Pi)<1$.
We remind ourselves that with these assumptions the symmetric generalized Lyapunov equation~\eqref{eq:GenLyap} has a unique solution $X$, and specifically it is symmetric and positive definite, see, e.g., \cite[Theorem~4.1]{DammDirectADI}.

Recall the fixed-point iteration
for the generalized Lyapunov equation~\eqref{eq:GenLyap},
\begin{align}\label{eq:fixed-point-iter}
\Lop(\hat X_{k+1}) = -\Pi(\hat X_{k}) - BB^T, \qquad k=0,1,\dots,
\end{align}
with $\hat X_0 = 0$.
The iteration has been presented as a convergent splitting in, e.g., \cite[Equation~(12)]{DammDirectADI}, \cite[Equation~(12)]{ZHANG2002}, and \cite[Equation~(4)]{Shank2016}.
We want to relate iteration~\eqref{eq:fixed-point-iter} to the $\mathcal{M}$-norm minimization problem. Hence consider the problem
\begin{align*}
\min_{\substack{\Delta\\\Delta=\Delta^T\succeq 0}} \|X - \hat X_{k} - \Delta\|^2_\mathcal{M}.
\end{align*}
A reason to restrict the minimization to symmetric positive semidefinite matrices is that we know that the solution $X=X^T\succeq 0$. Hence an iteration started with $X_0=0$ creates a sequence of symmetric positive semidefinite approximations, ordered as $\hat X_{k+1} \succeq \hat X_{k}$, and converging to the symmetric positive definite solution, cf. Section~\ref{sect:ALS} and \cite[Theorem~2.4]{Kressner:2015:Truncated}.
Naturally Proposition~\ref{prop:residual} gives us the solution in just one step. However, the computation is as difficult as the original problem and hence we consider minimizing an upper bound.
Similar as before we disregard the constant term $\|X-\hat X_k\|^2_\mathcal{M}$ in the minimization and consider
\begin{align*}
\min_{\substack{\Delta\\\Delta=\Delta^T\succeq 0}} \langle \Delta, \Delta\rangle_\mathcal{M} &- 2\trace( \Delta^T \R_{k})\\
&=
\min_{\substack{\Delta\\\Delta=\Delta^T\succeq 0}} \trace( \Delta^T(-\Lop(\Delta)-\Pi(\Delta))) - 2\trace( \Delta^T \R_{k})
\\&=
\min_{\substack{\Delta\\\Delta=\Delta^T\succeq 0}} \trace( \Delta^T(-\Lop(\Delta)-2\R_{k})) -\trace( \Delta^T\Pi(\Delta))
\\&\leq
\min_{\substack{\Delta\\\Delta=\Delta^T\succeq 0}} \trace( \Delta^T(-\Lop(\Delta)-2\R_{k})),
\end{align*}
where the inequality comes from that $\Delta^T$ and $\Pi(\Delta)$ are positive semidefinite matrices, and hence the trace is non-negative \cite{Neudecker:1992:AMatrixTrace}.
We want to say that the last step is minimized by $\Delta=-\Lop^{-1}(\R_{k})$, which is true if the residual is symmetric and positive semidefinite. We also want to show that this iteration is equivalent to the fixed-point iteration. The argument is closed by the following result.

\begin{theorem}\label{thm:fixed-point}
Consider the symmetric generalized Lyapunov equation \eqref{eq:GenLyap} with the additional assumptions that $A\prec 0$ and $\rho(\Lop^{-1}\Pi)<1$. Moreover, consider the sequence of approximations constructed as
  \begin{align}\label{eq:fixed-point-iter-residual}
  \begin{aligned}
  \hat X_0 &= 0\\
  \hat X_{k+1} &= \hat X_{k} - \Lop^{-1}(\R_{k}), \qquad k = 0,1,\dots,
  \end{aligned}
\end{align} 
where $\R_{k}$ is the residual associated with $\hat X_k$.
Then $\hat X_k=\hat X_k^T \succeq 0$ and $\R_k=\R_k^T \succeq 0$, for all $k\ge 0$. Furthermore, the iteration \eqref{eq:fixed-point-iter-residual} is equivalent to \eqref{eq:fixed-point-iter}.
\end{theorem}

\begin{proof}
The first part of the proof is by induction and similar to that of Theorem~\ref{thm:2.4_gen}. It trivially holds that $X_0=X_0^T\succeq0$ and $\R_0=\R_0^T\succeq 0$. Now assume that this is the case for some $k$.
Then $\Delta:=-\Lop^{-1}(\R_k)$ is symmetric and positive semidefinite, and hence $\hat X_{k+1}$ is symmetric and positive semidefinite. Moreover, since $\hat X_{k+1}$ is symmetric and because of the structure of the operators in \eqref{eq:GenLyap} we know that $\R_{k+1}$ is symmetric. Thus what is left to show for the first part is $\R_{k+1} \succeq0$.
We have that $\R_{k+1} \succeq0$ if and only if $y^T\R_{k+1}y\geq0$ for all $y\in\RR^{n}$. Hence take an arbitrary $y\in\RR^{n}$ and consider
\begin{align*}
y^T \R_{k+1}y
&=
y^T \R_k y + y^T\left(\Lop\left(\Delta\right) + \Pi\left(\Delta\right)\right)y
= y^T\left(\Pi\left(\Delta\right)\right)y \geq 0,
\end{align*}
where we in the first equality use the linearity of the operators and in the second the definition of $\Delta$. The last inequality holds since $\Delta$ is symmetric and positive semidefinite and $\Pi$ is a symmetric operator.

For the second part of the proof, what is left to show is that \eqref{eq:fixed-point-iter-residual} and \eqref{eq:fixed-point-iter} are equivalent. This comes directly from adding and subtracting $\Lop(\hat X_k)$ from the right-hand-side of \eqref{eq:fixed-point-iter}, and using the linearity of $\Lop$.
\qed
\end{proof}

\begin{corollary}
The fixed-point iteration \eqref{eq:fixed-point-iter} produces an increasing sequence of approximations $0=\hat X_0 \preceq \hat X_1\preceq\dots\preceq X$.
\end{corollary}

\begin{remark}
One could consider creating a subspace iteration from \eqref{eq:fixed-point-iter-residual}, by computing a few singular vectors of $\Lop^{-1}(\R_k)$ and add to the basis. The method seems to have nice convergence properties per iteration in the symmetric case, but not in the non-symmetric case. However, the (na{\"i}ve) computations are prohibitively expensive. See \cite{Shank2016} for a computationally more efficient way of exploiting the fixed-point iteration.
\end{remark}

\section{A residual-based rational Krylov generalization}\label{sect:Res_Rat_Kry}
Given Proposition~\ref{prop:residual} and the discussion in Sections~\ref{sect:H2_connection} and~\ref{sect:fixed-point} it seems as constructing an iteration based on the current residual can be a useful technique when solving the generalized Lyapunov equation.
Moreover, in \cite[Section~4]{Kressner:2015:Truncated} it is suggested that, so called, preconditioned residuals can be used in expanding the search space. It is further suggested that, for the Lyapunov equation, one such preconditioner could be a one-step-ADI preconditioner $P^{-1}_\text{ADI} = (A-\sigma I)^{-1}\kron(A-\sigma I)^{-1}$, for a suitable choice of the shift.
This strategy can be, somewhat, motivated by Theorem~\ref{thm:fixed-point}.

In this section we present a method that can be seen as a generalization of the rational Krylov subspace method. We further motivate why it can be natural to see it as a generalization of the rational Krylov subspace method, by considering the standard Lyapunov equation.

\subsection{Suggested search space}\label{sect:gen_space}
For the generalized Lyapunov equation \eqref{eq:GenLyap}, we suggest the following search space:
\begin{align}\label{eq:Rat-Kry-Res}
\Kspace_k := \vspan\{B,(A-\sigma_1 I)^{-1}r_0,(A-\sigma_2 I)^{-1}r_1,\dots,(A-\sigma_k I)^{-1}r_{k-1}\},
\end{align}
where $r_{k-1}$ is the most dominant left singular vector of the Galerkin residual $\R_{k-1}$ of $\Kspace_{k-1}$, and $\{\sigma_i\}_{i=1}^k$ is a sequence shift that needs to be chosen. In analogy with the discussion in \cite{Druskin.Simoncini.11}, we suggest that the shift is chosen iteratively, where the approximation error along the current direction is the largest. More precisely,
\begin{align}\label{eq:sigma_k1}
\sigma_{k} := \argmax_{\sigma\in[\sigma_\text{min},\sigma_\text{max}]}\left(\left\| r_{k-1} - (A - \sigma I)\Vbase_{k-1}(A_{k-1} - \sigma I)^{-1}\Vbase_{k-1}^T r_{k-1}\right\|\right),
\end{align}
where $\Vbase_{k-1}$ is a matrix with orthogonal columns containing basis of $\Kspace_{k-1}$, $A_{k-1} = \Vbase_{k-1}^T A \Vbase_{k-1}$, and $[\sigma_\text{min},\sigma_\text{max}]$ is a search interval. Typically for a stable matrix $A$, then $\sigma_\text{min}$ is the negative real part of the eigenvalue of $A$ with largest real part, and correspondingly $\sigma_\text{max}$ is the negative real part of the eigenvalue of $A$ with smallest real part.
Equations \eqref{eq:Rat-Kry-Res} and \eqref{eq:sigma_k1} can be straightforwardly incorporated in a Galerkin method for the generalized Lyapunov equation, the pseudocode is presented in Algorithm~\ref{alg:Rat-Kry-Res}.

\begin{algorithm} \label{alg:Rat-Kry-Res}
\caption{Residual-based rational Krylov type solver}
\SetKwInOut{Input}{input}\SetKwInOut{Output}{output}
\Input{$A,N_1,\dots,N_m\in\RR^{n\times n}$ $B\in\RR^{n\times r}$, $\tol$}
\Output{$\hat X$}
\BlankLine
\nl $\Vbase_0 = \emptyset$, $v_1 = \orth(B)$\\
\For{$k = 1,2,\dots $ until convergence}{
\nl $\Vbase_k = [\Vbase_{k-1}, v_k]$\\
\nl Compute the projected matrices: $A_k = \Vbase_k^T A \Vbase_k$, and $N_{i,k} = \Vbase_k^T N_i\Vbase_k$ for $i=1,2,\dots,m$, and $B_k = \Vbase_k^T B$ \\
\nl Solve the projected problem: $A_kY_k +Y_kA_k^T + \sum_{i=1}^m N_{i,k}Y_jN_{i,k}^T + B_kB_k^T = 0$\\
\nl Construct the (Galerkin) approximation: $\hat X_k = \Vbase_k Y_k \Vbase_k^T$ \label{alg_step:RatKry:approx}\\
\nl Compute the residual: $\R_k = A\hat X_k + \hat X_k A^T + \sum_{i=1}^m N_i\hat X_k N_i^T + BB^T$\label{alg_step:RatKry:residual}\\
\nl\If{$\|\R_k\|<\tol$}{
Break
}
\nl $r_k \gets $ the most dominant left singular vector of $\R_k$ \label{alg_step:RatKry:Direction-select}\\
\nl Select shift $\sigma_{k+1}$ according to \eqref{eq:sigma_k1} \label{alg_step:RatKry:Shift-select}\\
\nl $v_{k+1} = (A - \sigma_{k+1}I)^{-1} r_k$\\
\nl $v_{k+1} \gets$ orthogonalize $v_{k+1}$ with respect to $\Vbase_k$\\
}
\nl \Return $\hat X = \hat X_k$
\end{algorithm}
\begin{remark}\label{rem:iter_res_svd}
In practice the computation of the left singular vector can typically be done approximatively in an iterative fashion. This would also remove the need of computing the approximative solution $\hat X_k$ in Step~\ref{alg_step:RatKry:approx} and the residual in Step~\ref{alg_step:RatKry:residual} explicitly since the matrix vector product can be implemented as $\R_k v = A \Vbase_k Y_k \Vbase_k^T v + \Vbase_k Y_k \Vbase_k^T A^T v + \sum_{i=1}^m N_i \Vbase_k Y_k \Vbase_k^T N_i^T v + BB^Tv$. Such computations may, however, introduce inexactness and present a difficulty in a convergence analysis.
\end{remark}
\begin{remark}\label{rem:DS_shifts}
The dynamic shift-search in Step~\ref{alg_step:RatKry:Shift-select} can, heuristically, be changed for and analogue to \cite[(2.4) and (2.2)]{Druskin.Simoncini.11}. In our case we suggest
\begin{align}\label{eq:DS_shifts}
\sigma_k := \argmax_{\sigma\in\partial S} \frac{1}{|\mathfrak{r}_{k-1}(z)|},
\end{align}
where $S$ is approximates the mirrored spectrum of $A$ and  $\partial S$ is the boundary of $S$, and
\begin{align*}
\mathfrak{r}_{k-1}(z) := \frac{\prod_{j=1}^{\dim(\Kspace_{k-1})} z-\lambda_j^{(k-1)}}{\prod_{\ell=1}^{k-1}z-\sigma_\ell},
\end{align*}
with $\lambda_j^{(k-1)}$ being the Ritz values of $A_{k-1}$. Typically $S$ is approximated at each step using the convex hull of the Ritz values of $A_{k-1}$. It has been observed efficient in experiments since the maximization of \eqref{eq:DS_shifts} is computationally faster compared to \eqref{eq:sigma_k1}. For a practical comparison of convergence properties, see Section~\ref{sect:num}.
\end{remark}
\begin{remark}\label{rem:tangential}
The steps \ref{alg_step:RatKry:Direction-select}-\ref{alg_step:RatKry:Shift-select} in Algorithm~\ref{alg:Rat-Kry-Res} can be changed for a tangential-direction approach according to \cite{Druskin:2014:AdaptiveTangetial}. One practical way, although a heuristic, is to do the shift search according to either \eqref{eq:sigma_k1} or \eqref{eq:DS_shifts}, and then compute the principal direction(s) according to \cite[Section~3]{Druskin:2014:AdaptiveTangetial}, i.e., through a singular value decomposition of $\R_{k-1} - (A - \sigma_k I)\Vbase_{k-1}(A_{k-1} - \sigma_k I)^{-1}\Vbase_{k-1}^T \R_{k-1}$. It has been observed in experiments that such an approach tend to speed up the convergence, in terms of computation time,  since the computation of the residual is costly.
\end{remark}
\begin{remark}\label{rem:complex_shifts}
It is (sometimes) desirable to allow for complex conjugate shifts $\sigma_k$ and $\conj{\sigma}_k$, although, for reasons due to computations and model interpretation, one wants to keep the basis real. This is achievable also in the proposed setting by observing the following standard relation 
$
\vspan\left\{(A-\sigma_k I)^{-1}r_{k-1},\,(A-\conj{\sigma}_kI)^{-1}r_{k-1}\right\}
= 
\vspan\left\{\re((A-\sigma_k I)^{-1}r_{k-1}),\,\im((A-\sigma_k I)^{-1}r_{k-1})\right\}$. 
Although it requires two shifts to be used together with the vector $r_{k-1}$.
\end{remark}

\subsection{Analogies to the linear case}\label{sect:lin}
To give further motivation to the suggested subspace in \eqref{eq:Rat-Kry-Res}, we draw parallels with the (standard) rational Krylov subspace for the \emph{standard Lyapunov equation}. The reasoning in this section can be compared to \cite[Section~2]{Baars:2017:Continuation}.
We consider the equation
\begin{align}\label{eq:Lyap}
\Lop(X) + bb^T = 0,
\end{align}
where $\Lop$ is defined by \eqref{eq:L} and $b\in\RR^{n}$. Note that Definitions~\ref{def:Gal_approx} and \ref{def:Gal_residual} are analogous for the (standard) Lyapunov equation \eqref{eq:Lyap}, but with $\Pi=0$. To prove the main result of this section we need the following lemma.

\begin{lemma}\label{lem:Res_Km}
Let $A\in\RR^{n\times n}$ and $s_a\in\RR$ be any scalar such that $(A-s_{a}I)$ is nonsingular. Moreover, let $\Vbase\in\RR^{n\times k}$, $k\leq n$, be orthogonal, i.e., $\Vbase^T\Vbase = I$, and let $\R\in\RR^{n\times n}$ be such that $\vrange((A-s_a I)^{-1}\R)\subseteq\vspan(\Vbase)$.
Then $\R = (A-s_aI)\Vbase(\Vbase^T A \Vbase - s_aI)^{-1}\Vbase^T\R$.
\end{lemma}
\begin{proof}
To prove the statement we consider the right-hand-side of the asserted equality,
{
\begin{align*}
  (A - s_{a}I)\Vbase(\Vbase^T A \Vbase &- s_{a}I)^{-1}\Vbase^T \R\\
  &=
(A - s_{a}I)\Vbase(\Vbase^T A \Vbase - s_{a}I)^{-1}\Vbase^T (A-s_{a}I)(A-s_{a}I)^{-1}\R\\
&= (A - s_{a}I)\Vbase(\Vbase^T A \Vbase - s_{a}I)^{-1}\Vbase^T (A-s_{a}I)\Vbase\Vbase^T(A-s_{a}I)^{-1}\R,
\end{align*}}%
where the first equality is just an insertion of an identity matrix, and the second equality follows from the assumption that $\vrange((A-s_{a}I)^{-1}\R)\subseteq\vspan(\Vbase)$. By observing that $(\Vbase^T A \Vbase - s_{a}I)^{-1}\Vbase^T (A-s_{a}I)\Vbase = I$ the expression can be further simplified as
{
\begin{align*}
 (A - s_{a}I)\Vbase(\Vbase^T A \Vbase - s_{a}I)^{-1}\Vbase^T \R
&= (A - s_{a}I)\Vbase\Vbase^T(A-s_{a}I)^{-1}\R\\
&= (A - s_{a}I)(A-s_{a}I)^{-1}\R = \R,
\end{align*}}%
where the second equality, again, follows from $\vrange((A-s_{a}I)^{-1}\R)\subseteq\vspan(\Vbase)$.
\qed
\end{proof}

\begin{theorem}\label{thm:linear_residual_rat_Krylov}
Let $A\in\RR^{n\times n}$, $b\in\RR^n$, and let $\{s_i\}_{i=1}^{k+1}$ be a sequence of shifts such that $A-s_i I$ is nonsingular, and define $\Kspace_k := \vspan\{b,(A-s_1 I)^{-1}b,\dots,(A-s_kI)^{-1}b\}$, and $\Kspace_{k+1}$ analogously. Let $\Vbase_k$ be a basis of $\Kspace_k$, $\Vbase_{k+1}$ a basis of $\Kspace_{k+1}$, and let $v_{k+1}\in\RR^n$ be such that $\Vbase_{k+1} = \begin{pmatrix}
\Vbase_k & v_{k+1} \end{pmatrix}$. Moreover, let $\R_k\in\RR^{n\times n}$ be the Galerkin residual with respect to \eqref{eq:Lyap}.
Then each column of $(A-s_{k+1}I)^{-1} \R_k$ is in $\vspan(V_{k+1})$, i.e., $\vrange((A-s_{k+1}I)^{-1} \R_k)\subseteq\vspan(V_{k+1})$. Furthermore, if $\vrange((A-s_{k+1}I)^{-1} \R_k)\subseteq\vspan(V_{k})$, then $\R_k = 0$.
\end{theorem}
\begin{proof}
We start by proving the first claim. From \cite[Proposition~4.2]{Druskin.Simoncini.11} we have that $\R_k = UJU^T$ with
\begin{align*}
J = \begin{pmatrix}
0 & 1\\ 1 & 0
\end{pmatrix}
\qquad\text{ and }\qquad 
U = \begin{pmatrix}
u_1 & u_2
\end{pmatrix},
\end{align*}
with $u_1 := V_kY_kH_k^{-T}e_{k}h_{k+1,k}$ and $u_2 := v_{k+1}s_{k+1}-(I-V_kV_k^T)Av_{k+1}$.
From the definition of $u_1$ one can see that $u_1 = V_k\alpha$ for some $\alpha\in\RR^n$, hence $u_1\in\Kspace_m$, and thus $(A-s_{m+1}I)^{-1}u_1\in\Kspace_{m+1}$. 
Similarly, from the definition of $u_2$ one can see that $u_2 = (s_{k+1}I-A)v_{k+1} + V_kV_k^T A v_{k+1} = (s_{k+1}I-A)v_{k+1} + V_k\beta$, for some $\beta\in\RR^n$. Thus $(A-s_{k+1}I)^{-1}u_2 = -v_{k+1} + (A-s_{k+1}I)^{-1}V_k\beta\in\Kspace_{k+1}$.
Hence $\vrange((A-s_{k+1}I)^{-1}\R_k)\subseteq\Kspace_{k+1}$, which proves the first claim.

To prove the second claim we assume that $\vrange((A-s_{m+1}I)^{-1}\R_m)\subseteq\Kspace_{k}=\vspan(\Vbase_k)$. Under this assumption we can use Lemma~\ref{lem:Res_Km} and the fact that $\R_k = \R_k^T$ to get
\begin{align*}
\R_k
&= (A - s_{k+1}I)\Vbase_m(\Vbase_k^TA\Vbase_k - s_{k+1}I)^{-1}\Vbase_k^T\R_k\\
&= (A - s_{k+1}I)\Vbase_m(\Vbase_k^TA\Vbase_k - s_{k+1}I)^{-1}\Vbase_k^T\R_k\Vbase_k(\Vbase_k^TA^T\Vbase_k-s_{k+1}I)^{-1}\Vbase_k^T(A^T - s_{k+1}I)\\
&= 0,
\end{align*}
since $\R_k$ is the Galerkin residual and thus $\Vbase_k^T\R_k\Vbase_k=0$.
\qed
\end{proof}

The main observation and conclusion from Theorem~\ref{thm:linear_residual_rat_Krylov} can be phrased as follows: Consider the two spaces $\Kspace_k := \vspan\{b,(A-s_1 I)^{-1}b,\dots,(A-s_kI)^{-1}b\}$ and $\hat\Kspace_k := \vspan\{\R_{-1},(A-s_1 I)^{-1}\R_0,\dots,(A-s_kI)^{-1}\R_{k-1}\}$, where $\R_{-1} = b$ and $\R_i$ is the Galerkin residual in space $\Kspace_{i}$, $i=0,1,\dots,k-1$ .
Then for for all relevant cases, i.e., $\R_i\neq 0$ for $i=-1,0,\dots,k-1$, we have that $\Kspace_k = \hat\Kspace_k$. Hence the suggested subspace in \eqref{eq:Rat-Kry-Res} can be seen as a natural generalization of a Rational Krylov subspace for linear matrix equations.

\begin{remark}
The result in Theorem~\ref{thm:linear_residual_rat_Krylov} generalizes naturally to the case $B\in\RR^{n\times r}$, where the arguments have to be done with block matrices. As noted in \cite[just before Section~4.1]{Druskin.Simoncini.11} the changes when going to blocks are mostly technical. Results needed to generalize \cite[Proposition~4.2]{Druskin.Simoncini.11} are found in, e.g., \cite{Abidi:2016:Adaptive}, and the block case is implemented in the code available at Simoncini's webpage\footnote{\url{http://www.dm.unibo.it/~simoncin/software.html}}.
\end{remark}

\section{Numerical examples}\label{sect:num}
In this section we numerically compare different methods discussed in the paper. All algorithms are treated in a subspace fashion and we compare practically achieved approximation properties as a function of subspace dimension.
For small and moderate sized problems there are algorithms for computing the full solution, although costly, see \cite[Algorithm~2]{jarlebring2017krylov}, cf. \cite[equation~(12)]{ZHANG2002}. Nevertheless this allows for inspection of the actual relative error, i.e., 
\begin{align*}
\|X-\hat X_k\|/\|X\|,
\end{align*} where $\hat X_k$ is the approximation and $X$ is the exact solution. Moreover, it also allows comparison with the (in the Frobenius norm) optimal low-rank approximation based on the SVD.

We summarize some of the implementation details. Specifically, BIRKA is implemented as described in Algorithm~\ref{alg:BIRKA}, with a maximum allowed number of iterations equal to 100. Convergence tolerance is implemented as relative norm difference of the vector of sorted eigenvalues and was set to $10^{-3}$. Each subspace is computed separately since the algorithm does not have a natural way to extend the subspace dimension based on a smaller subspace.
The complete method based on ALS is a subspace method along the lines of \cite[Algorithm~3]{Kressner:2015:Truncated}, rather than an iteratively updated method as described in Theorem~\ref{thm:2.4_gen}. Moreover, because of the structure of the generalized Lyapunov equation, the solution is symmetric even if the coefficient matrices are not. Hence we use a symmetric version of ALS even for the non-symmetric examples. The convergence tolerance is implemented analogously to BIRKA, but here it is only a scalar each time. The maximum allowed number of iterations in ALS was set to 20, and the tolerance was set to $10^{-2}$.
Regarding rational-Krylov-type methods there are a plethora of variants to choose from. We compare the following methods, which we give short labels for the legends further down:
\begin{itemize}
\item A: $\mathcal{K}_k$ as in \eqref{eq:Rat-Kry-Res}, according to Algorithm~\ref{alg:Rat-Kry-Res}
\item B: Algorithm~\ref{alg:Rat-Kry-Res} but with tangential directions according to Remark~\ref{rem:tangential}, though with shifts according to \eqref{eq:sigma_k1}
\item C: Algorithm~\ref{alg:Rat-Kry-Res} but with shifts according to \eqref{eq:DS_shifts}
\item D: Algorithm~\ref{alg:Rat-Kry-Res} but with tangential directions according to Remark~\ref{rem:tangential} and shifts according to \eqref{eq:DS_shifts}
\item E: Standard Rational Krylov. More precisely, $\mathcal{K}_k$ similar to \eqref{eq:Rat-Kry-Res} but instead of using the singular vector of the residual, $r_{k-1}$, we use the right-hand-side $B$ in both \eqref{eq:Rat-Kry-Res} and \eqref{eq:sigma_k1}
\item F: A: $\mathcal{K}_k$ as in \eqref{eq:Rat-Kry-Res}, but with on-beforehand-prescribed shifts given as the recycling of mirrored eigenvalues from a size-10-BIRKA (convergence tolerance set to $10^{-3}$). Mirrored eigenvalues are potentially complex, with positive real part and taken in ascending order according to the real parts.
\end{itemize}
For C, D, and F the shifts may be complex-valued, and the complex arithmetic is avoided by creating the space in accordance with Remark~\ref{rem:complex_shifts}.
For the shift-computing versions of the rational-Krylov-type algorithms, the maximum and minimum eigenvalues  were computed and the mirrored shift-search-boundaries were perturbed with the factors $1.01$ and $0.99$ for the upper and lower boundaries respectively, as to slightly enlarge the region.
Orthogonalization of the basis is implemented very similar to how the SVD-based orthogonalization is done in \matlab. More precisely, it is using \matlab\ built-in QR factorization and keep vectors only if the corresponding diagonal element in R is large enough.
Implementation for the methods A-F is available online.\footnote{\url{https://people.kth.se/~eringh/software/res_rat_Kry_type/}} 
The simulations were done in \matlab\ R2018a (9.4.0.813654) on a computer with four 1.6 GHz processors and 16 GB of RAM.

We test the algorithms on three different problems. In all problems we are computing an approximation to an associated controllability Gramian to a bilinear control system, as in \eqref{eq:bil_gramians}. The examples all have stable Lyapunov operators. The first example is symmetric, the second is non-symmetric but symmetrizable, and the third example non-symmetric. 

\subsection{Heat equation}

The first example is motivated by an optimal control problem for the selective cooling of steel profiles, see \cite{EppT01}. In this example, the state variable $y$ models the evolution of a temperature and is described by a two-dimensional heat equation,
\begin{align*}
\frac{\partial}{\partial t} w(x,y,t) &= \Delta w(x,y,t)  \qquad\quad &&(x,y,t)\in(0,1)\times(0,1)\times(0,T)\\
w(x,y,0)&=w_0(x,y) && x\in(0,1)\\
-\frac{\partial}{\partial x} w(0,y,t)&= 0.5 (w(0,y,t) - 1) u(t) && (y,t) \in (0,1)\times(0,T)\\
w(1,y,t)&=0 && (y,t) \in (0,1)\times(0,T)\\
w(x,0,t)&=0 && (x,t) \in (0,1)\times(0,T)\\
w(x,1,t)&=0 && (x,t) \in (0,1)\times(0,T),
\end{align*}
where a control $u(t)$ enters bilinearly  from the left boundary through a Robin boundary condition, and the other spatial boundaries satisfy homogeneous Dirichlet conditions. The control can be interpreted as the spraying intensity of a cooling fluid. The system is discretized in space using centred finite difference, which yields a bilinear system with
$A\in\RR^{5041\times 5041}$, $B\in\RR^{5041}$, $m=1$, and $N_1=N\in\RR^{5041\times 5041}$. 
It can be further noted that, $A=A^T\prec0$ and $N=N^T$, and hence the theory of $\H2$-optimality and the definition of the $\mathcal{M}$-norm, \eqref{eq:M-norm}, from above is applicable.

\begin{figure}
  \centering
  \includegraphics{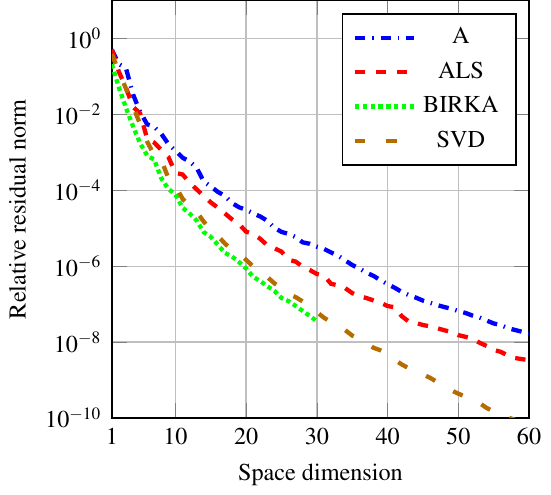}
  \hfill
  \centering
  \includegraphics{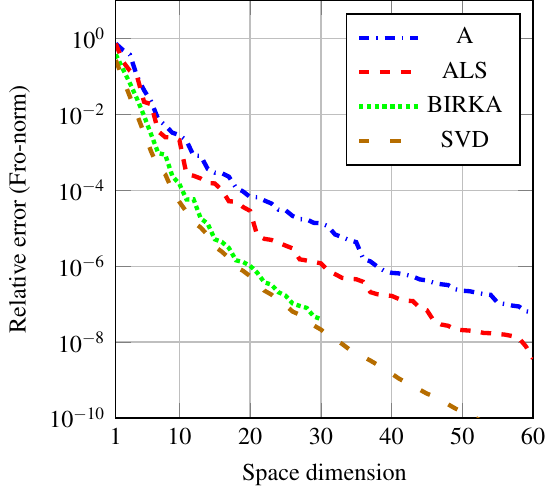}

   \caption{Cross-algorithm comparison for the heat equation. Relative residual norm (left), relative error in Frobenius norm (right).}\label{fig:heat_5041_comp}
$\text{ }$\\
  \centering
  \includegraphics{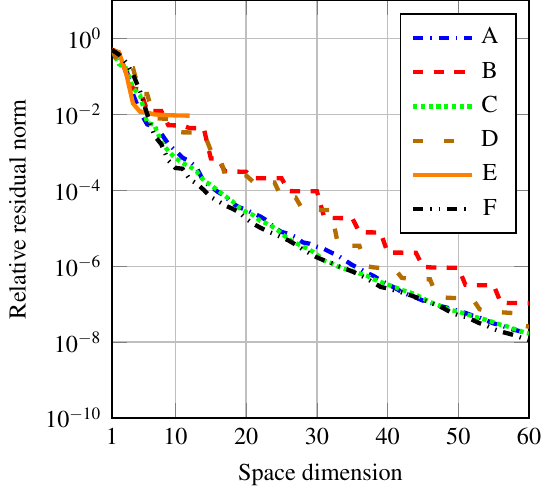}
  \hfill
  \centering
  \includegraphics{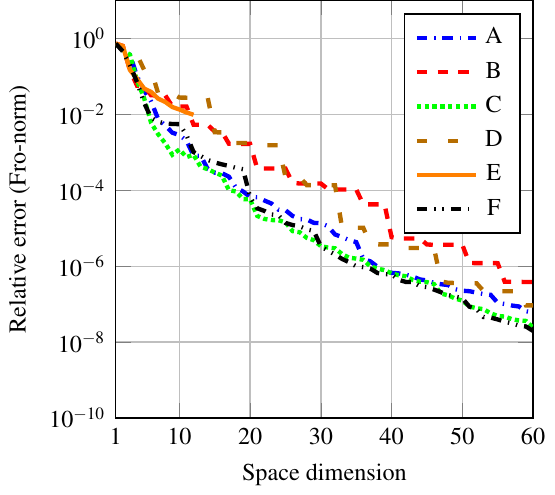}
   \caption{Rational-Krylov-type method comparison for the heat equation. Relative residual norm (left), relative error in Frobenius norm (right). Compare with Figure~\ref{fig:heat_5041_comp} as the lines for method A are the same in both figures. For a description of the labels, see the beginning of this section.}\label{fig:heat_5041_kry}
$\text{ }$\\
  \centering
  \includegraphics{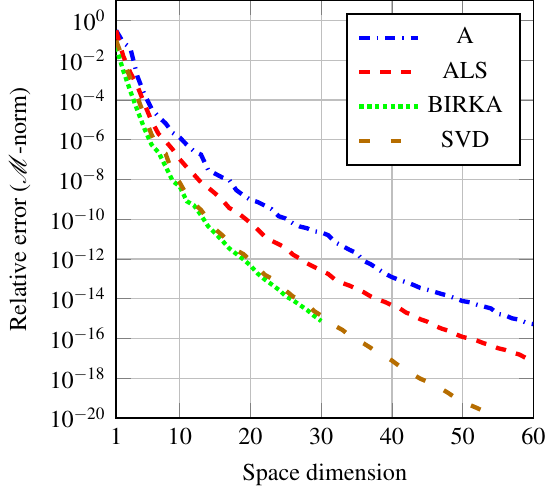}
  \hfill
  \centering
  \includegraphics{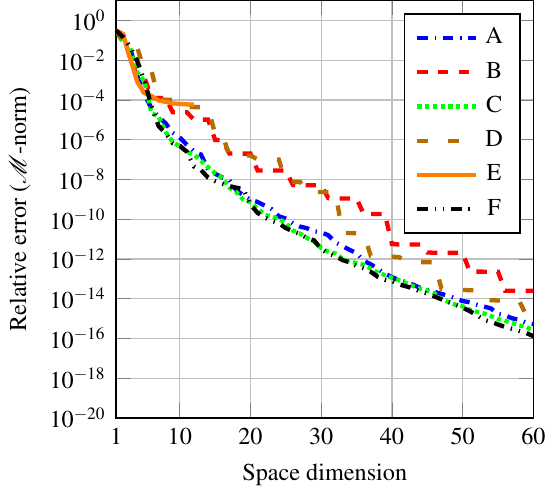}
   \caption{The cross-algorithm comparison (left) and rational-Krylov-type methods (right) for the heat equation. However, the relative error is measured in the $\mathcal{M}$-norm. The lines for method A are the same in both plots.}\label{fig:heat_5041_enorm}
\end{figure}

We compare different methods discussed in the paper, and we compare both the relative residual norm, as well as the relative error. For readability the plots have been split, hence in Figure~\ref{fig:heat_5041_comp} we compare across different methods, and in Figure~\ref{fig:heat_5041_kry} we compare between different flavours the rational-Krylov-type methods.
It can be observed, see Figure~\ref{fig:heat_5041_comp}, that for this problem BIRKA has extremely good performance, even outperforming the SVD in relative residual norm. Nevertheless, the larger BIRKA subspaces are rather costly to compute. In comparison ALS shows good performance compared to the rational-Krylov-type subspace, and is rather cheap to compute.
When comparing the different Rational-Krylov-type methods, see Figure~\ref{fig:heat_5041_kry}, we see that standard rational-Krylov (E) has the problem that it reaches an invariant subspace and is unable to extend larger than dimension 13. However, stagnation has already been observed. The methods A, C, and F are similar, although B and D are only slightly worse in the error per subspace dimension comparison and practically sometimes even faster to compute.

Since the $\mathcal{M}$-norm is defined for this problem we compare the relative error also in this norm, see Figure~\ref{fig:heat_5041_enorm}. The trend is similar as in the Frobenius norm, although it can be observed that in general the error is smaller and BIRKA has best performance, even compared to the SVD (consider that this is relative error, however, in another norm).

\subsection{1D Fokker-Planck}
The second example is from quantum physics, where a one-dimensional Fokker-Planck equation describes the evolution of a probability density function, $\rho$, of a particle affected by a potential. Parts of the potential can be manipulated by a so-called optical tweezer, which constitutes our control. For more details see \cite{Hartmann:2013:Balanced}.
More precisely, we consider
\begin{align*}
\frac{\partial}{\partial t} \rho(x,t)  &= \nu\frac{\partial^2}{\partial x^2} \rho(x,t)
+ \frac{\partial}{\partial x}\left(\rho(x,t)\frac{\partial}{\partial x}V(x,t)\right) 
\quad &&(x,t)\in(-6,6)\times(0,T)\\
\rho(x,0)&=\rho_0(x) && x\in(-6,6)\\
\nu \frac{\partial}{\partial x}\rho(x,t)&=-\rho(x,t)\frac{\partial}{\partial x}V(x,t) &&(x, t) \in \{-6,6\} \times (0,T),
\end{align*}
where the potential is $V(x,t) = W(x) + \alpha(x) u(t)$, with the ground (fixed) potential being $W(x) = \left(\left((0.5x^2-15)x^2 + 199\right)x^2 + 28x + 50\right)/200$, and $\alpha(x)$ is an approximately linear control shape function, see \cite{Breiten:2017:Numerical} for further details. In a weighted inner product, the dynamics can be described by self-adjoint operators. However, here we employ an upwinding type finite difference scheme with $5000$ grid point, leading to a non-symmetric system. As has been pointed out in \cite{Hartmann:2013:Balanced}, the system matrix $A$ is not asymptotically stable due to a simple zero eigenvalue associated with the stationary probability distribution. Using a projection-based decoupling, it is however possible to work with an asymptotically stable system of dimension $n=4999$.  Similar to the first example, our control variable is a scalar and, consequently, we only obtain a single bilinear coupling matrix $N_1=N$.  Since the system is non-symmetric, the operator $\mathcal{M}$ is generally indefinite and hence we make no comparisons in the $\mathcal{M}$-norm.

\begin{figure}
  \centering
  \includegraphics{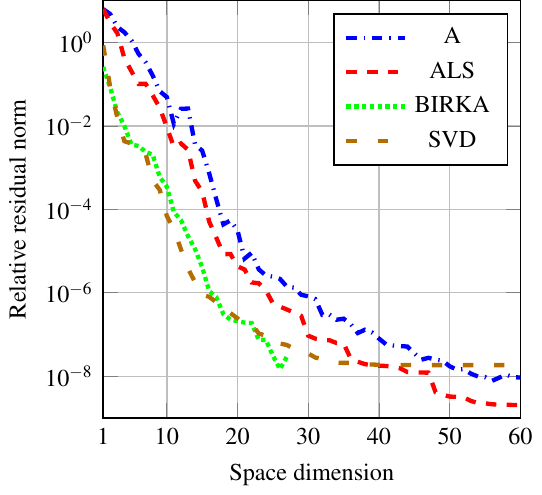}
  \hfill
  \centering
  \includegraphics{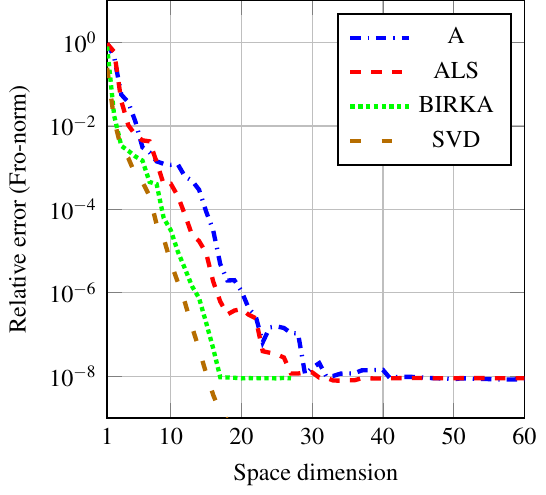}
   \caption{Cross-algorithm comparison for 1D Fokker-Planck equation. Relative residual norm (left) and relative error (right).}\label{fig:fp_4999_comp}
$\text{ }$\\
  \centering
  \includegraphics{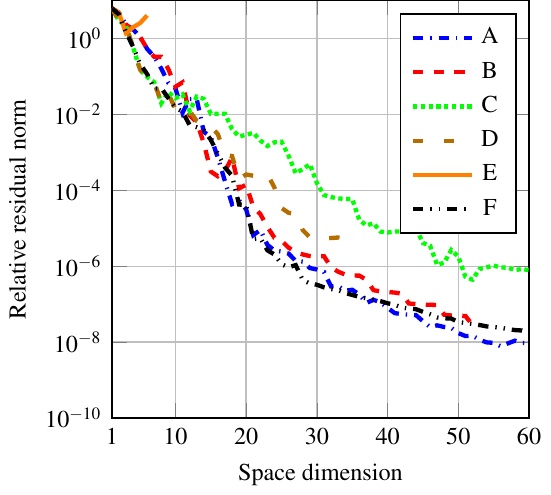}
  \hfill
  \centering
  \includegraphics{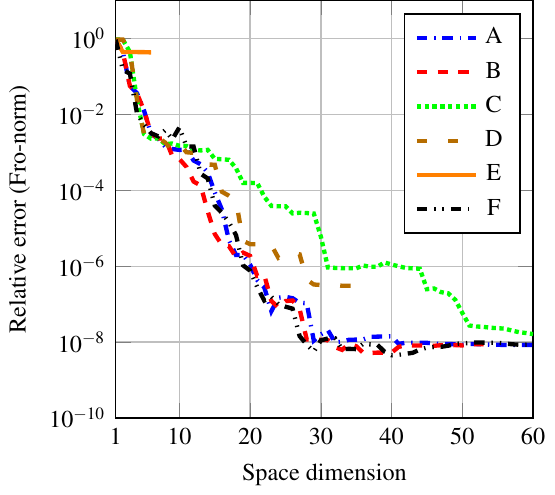}
   \caption{Rational-Krylov-type method comparison for 1D Fokker-Planck equation. Compare with Figure~\ref{fig:fp_4999_comp} as the lines for method A are the same in both figures. For a description of the labels, see the beggnining of this section.}\label{fig:fp_4999_kry}
\end{figure}

The plots in Figures~\ref{fig:fp_4999_comp} and \ref{fig:fp_4999_kry} are analogous to the plots in Figures~\ref{fig:heat_5041_comp} and \ref{fig:heat_5041_kry} respectively. However, for this example the direct solver stagnated at a relative residual of about $10^{-8}$, which can be seen in the stagnation of the SVD approximation in the left of Figure~\ref{fig:fp_4999_comp}. As a result, the comparisons of relative error performance, the right of Figures~\ref{fig:fp_4999_comp} and \ref{fig:fp_4999_kry}, has an artificial stagnation around $10^{-8}$. At this level the convergence stagnates since it measures the discrepancy between the method-approximations and the inexact reference solution, rather than the true error of the method-approximations. Nevertheless we believe the comparisons to be fair more or less upto to point of stagnation, which is further justified by the relative residual plots showing similar behaviour. However, the relative residual also indicating stagnation around $10^{-8}$ for the other methods as well, although not quite as clear as for the SVD.

From Figure~\ref{fig:fp_4999_comp} we see the BIRKA performs well for this problem. However, the subspaces of dimension 28 and 29 did not converge in a 100 iterations and hence these, as well as iteration 30, are left out of the plots. This illustrates a drawback with the method. The performance difference between ALS and the rational-Krylov-type method is slightly smaller compared to the previous example. Among the rational-Krylov-type methods A, B, and F seems to have similar performance, whereas C is clearly worse. Method E has the same problem as in the previous example, and method D also ends up with an insufficient subspace.

\subsection{Burgers equation}
\begin{figure}
  \centering
  \includegraphics{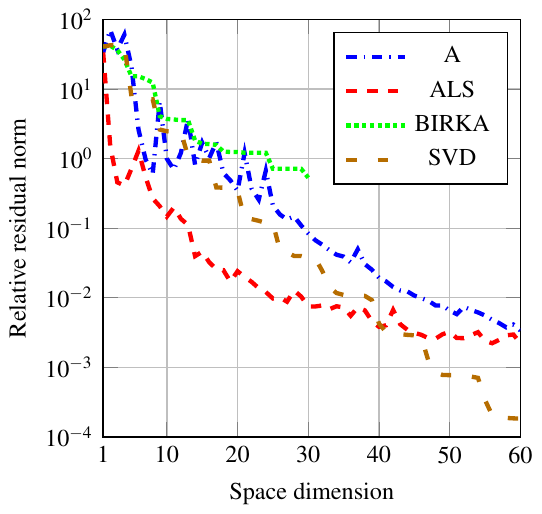}
  \hfill
  \centering
  \includegraphics{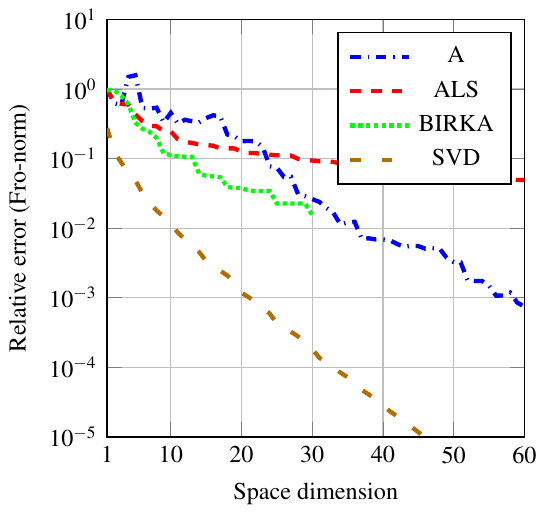}
   \caption{Cross-algorithm comparison for Burgers equation. Relative residual norm (left) and relative error (right).}\label{fig:bu_5112_comp}
$\text{ }$\\
  \centering
  \includegraphics{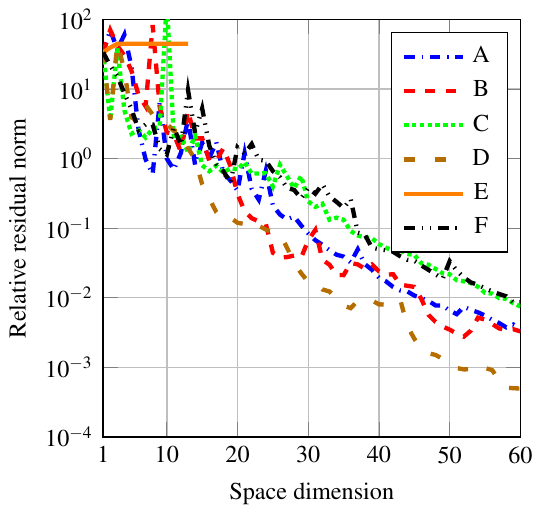}
  \hfill
  \centering
  \includegraphics{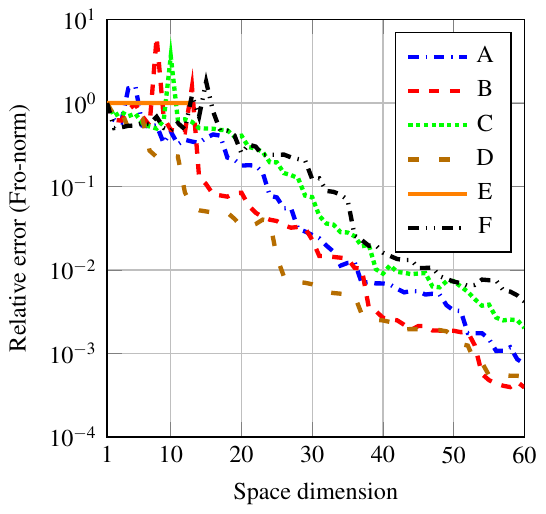}
   \caption{Rational-Krylov-type method comparison for Burgers equation. Compare with Figure~\ref{fig:bu_5112_comp} as the lines for method A are the same in both figures. For a description of the labels, see the beggnining of this section.}\label{fig:bu_5112_kry}
\end{figure}
In the third example we consider an approximation to the one-dimensional viscous Burgers equation
\begin{align*}
\frac{\partial}{\partial t} w(x,t) + w(x,t) \frac{\partial}{\partial x} w(x,t) &= \nu \frac{\partial^2}{\partial x^2}w(x,t)   \qquad\quad &&(x,t)\in(0,1)\times(0,T)\\
w(x,0)&=w_0(x) && x\in(0,1)\\
w(0,t)&=u(t) && t \in (0,T)\\
w(1,t)&=0 && t \in (0,T),
\end{align*}
where $\nu=0.1$ is constant. The solution $w(x,t)$ can be interpreted as a velocity and the equation occurs in, e.g., modelling of gas or traffic flow. A control input $u(t)$ is applied at the left boundary. The problem is discretized in space using centered finite differences with 71 uniformly distributed grid points. Using a second order Carleman bilinearization, we obtain a bilinear control system approximation with $A,N\in\RR^{5112\times 5112}$ and $B\in\RR^{5112}$, see \cite{Breiten:2010:Krylov} for further details. 
Note that in this case $A$ is an asymptotically stable but non-symmetric matrix. In order to ensure the positive semidefiniteness of the Gramian, we scale the control matrices $N$ and $B$ with a factor $\alpha=0.25$. We emphasize that the control law is scaled proportionally with $\frac{1}{\alpha}$ such that the dynamics remain unchanged, for further discussion see \cite[Section~3.4]{DammBenner}.

The comparison is similar to the previous examples and the Figures~\ref{fig:bu_5112_comp} and \ref{fig:bu_5112_kry} are analogous to the respective Figures~\ref{fig:heat_5041_comp} and \ref{fig:heat_5041_kry}. The problem is difficult in the sense that the singular values of the solution decays slowly. Moreover, the direct method stagnates at a relative residual norm of $5\cdot10^{-6}$. This is, however, less visible compared to the previous example since in general the convergence is slower.

For this example the performance of BIRKA is not significantly better than other methods, which is not surprising since the theoretical justifications for the method are not valid. ALS shows faster convergence in relative residual norm but slower convergence in relative error, as well as indications of stagnation. However, the theoretical justifications for ALS are also not valid for this problem and the result is in line with the results in \cite{Kressner:2015:Truncated}. Regarding the rational-Krylov-type methods it seems as if method D and B has good performance, and method E is not working for this problem either.

\section{Conclusions and outlooks}\label{sect:concl}
In this paper we have studied iterative methods for computing approximations to the generalized Lyapunov equation. 
The methods have been studied from an energy-norm perspective, as well as a model-reduction perspective, with connections made in between. 
Common for all methods studied is that they use the current residual in the iterations. Computing the residual can in itself be costly for a truly large scale problem, although approximate dominant directions can be computed in an iterative fashion, resulting in an inner-outer-type iteration. However, more research is needed to understand the consequences of such inexact subspaces. Moreover, we have proposed a rational-Krylov-type subspace for solving the generalized Lyapunov equation. Simulations indicate competitive performance, at least in the non-symmetric case where optimality statements for the other methods are no longer valid. Simulations show that methods~A and~F do good or decently good for all three examples. The ALS iteration, as well as results from the literature, cf. \cite{Ahmad2017}, 
seems to indicate that subspaces of the type $(A-\sigma I -\mu N)^{-1}B$ could be useful. Although we have not been able to exploit this efficiently. More research is needed to understand the theoretical aspects of the suggested, and related, spaces.

\section*{Acknowledgment}
This research started when the second author visited the first author at the Karl-Franzens-Universit{\"a}t in Graz; the kind hospitality was greatly appreciated. The visit was made possible due to the generous support from the European Model Reduction Network (COST action TD1307, STSM grant 38025).


\bibliographystyle{spmpsci}
\bibliography{manuscript}

\begin{thebibliography}{10}
\providecommand{\url}[1]{{#1}}
\providecommand{\urlprefix}{URL }
\expandafter\ifx\csname urlstyle\endcsname\relax
  \providecommand{\doi}[1]{DOI~\discretionary{}{}{}#1}\else
  \providecommand{\doi}{DOI~\discretionary{}{}{}\begingroup
  \urlstyle{rm}\Url}\fi

\bibitem{Abidi:2016:Adaptive}
Abidi, O., Hached, M., Jbilou, K.: Adaptive rational block {A}rnoldi methods
  for model reductions in large-scale {MIMO} dynamical systems.
\newblock N Tren Math \textbf{4}(2), 227--239 (2016)

\bibitem{Ahmad2017}
Ahmad, M., Baur, U., Benner, P.: Implicit {V}olterra series interpolation for
  model reduction of bilinear systems.
\newblock J. Comput. Appl. Math. \textbf{316}(Supplement C), 15 -- 28 (2017).
\newblock Selected Papers from NUMDIFF-14

\bibitem{Al-Baiyat:1993:ANew}
Al-Baiyat, S.A., Bettayeb, M.: A new model reduction scheme for k-power
  bilinear systems.
\newblock In: Proceedings of 32nd {IEEE} Conference on Decision and Control,
  vol.~1, pp. 22--27 (1993).
\newblock \doi{10.1109/CDC.1993.325196}

\bibitem{Baars:2017:Continuation}
Baars, S., Viebahn, J., Mulder, T., Kuehn, C., Wubs, F., Dijkstra, H.:
  Continuation of probability density functions using a generalized {L}yapunov
  approach.
\newblock J. Comput. Phys. \textbf{336}, 627 -- 643 (2017).
\newblock \doi{https://doi.org/10.1016/j.jcp.2017.02.021}

\bibitem{Bartels1972}
Bartels, R.H., Stewart, G.W.: Algorithm 432: {S}olution of the {M}atrix
  {E}quation {$AX+XB=C$}.
\newblock Comm. {ACM} \textbf{15}, 820--826 (1972)

\bibitem{Benner2012}
Benner, P., Breiten, T.: Interpolation-based {$\mathcal{H}_2$}-model reduction
  of bilinear control systems.
\newblock SIAM J. Matrix Anal. Appl. \textbf{33}(3), 859--885 (2012)

\bibitem{Benner:2013:low}
Benner, P., Breiten, T.: Low rank methods for a class of generalized {L}yapunov
  equations and related issues.
\newblock Numer. Math. \textbf{124}(3), 441--470 (2013)

\bibitem{Benner:2014:optimality}
Benner, P., Breiten, T.: On optimality of approximate low rank solutions of
  large-scale matrix equations.
\newblock Systems Control Lett. \textbf{67}, 55--64 (2014)

\bibitem{DammBenner}
Benner, P., Damm, T.: {Lyapunov} equations, energy functionals, and model order
  reduction of bilinear and stochastic systems.
\newblock SIAM J. Control Optim. \textbf{49}(2), 686--711 (2011)

\bibitem{Breiten:2010:Krylov}
Breiten, T., Damm, T.: {K}rylov subspace methods for model order reduction of
  bilinear control systems.
\newblock Systems Control Lett. \textbf{59}(8), 443 -- 450 (2010)

\bibitem{Breiten:2017:Numerical}
{Breiten}, T., {Kunisch}, K., {Pfeiffer}, L.: Numerical study of polynomial
  feedback laws for a bilinear control problem.
\newblock ArXiv e-prints  (2017).
\newblock 1709.04227

\bibitem{DammDirectADI}
Damm, T.: Direct methods and {ADI}-preconditioned {Krylov} subspace methods for
  generalized {Lyapunov} equations.
\newblock Numer. Linear Algebra Appl. \textbf{15}(9), 853--871 (2008)

\bibitem{Druskin2009}
Druskin, V., Knizhnerman, L., Zaslavsky, M.: Solution of large scale
  evolutionary problems using rational {K}rylov subspaces with optimized
  shifts.
\newblock SIAM J. Sci. Comput. \textbf{31}(5), 3760--3780 (2009).
\newblock \doi{10.1137/080742403}

\bibitem{Druskin2010}
Druskin, V., Lieberman, C., Zaslavsky, M.: On adaptive choice of shifts in
  rational {K}rylov subspace reduction of evolutionary problems.
\newblock SIAM J. Sci. Comput. \textbf{32}(5), 2485--2496 (2010).
\newblock \doi{10.1137/090774082}

\bibitem{Druskin.Simoncini.11}
Druskin, V., Simoncini, V.: Adaptive rational {K}rylov subspaces for
  large-scale dynamical systems.
\newblock Systems Control Lett. \textbf{60}(8), 546--560 (2011)

\bibitem{Druskin:2014:AdaptiveTangetial}
Druskin, V., Simoncini, V., Zaslavsky, M.: Adaptive tangential interpolation in
  rational {K}rylov subspaces for {MIMO} dynamical systems.
\newblock SIAM J. Matrix Anal. Appl. \textbf{35}(2), 476--498 (2014).
\newblock \doi{10.1137/120898784}

\bibitem{EppT01}
Eppler, K., Tr{\"o}ltzsch, F.: Fast optimization methods in the selective
  cooling of steel.
\newblock In: M.~Gr{\"o}tschel, S.~Krumke, J.~Rambau (eds.) Online Optimization
  of Large Scale Systems, pp. 185--204. Springer Berlin Heidelberg (2001)

\bibitem{Flagg:2012:Convergence}
Flagg, G., Beattie, C., Gugercin, S.: Convergence of the iterative rational
  krylov algorithm.
\newblock Systems Control Lett. \textbf{61}(6), 688 -- 691 (2012).
\newblock \doi{https://doi.org/10.1016/j.sysconle.2012.03.005}

\bibitem{Flagg2015}
Flagg, G., Gugercin, S.: Multipoint {V}olterra series interpolation and
  {$\mathcal{H}_2$} optimal model reduction of bilinear systems.
\newblock SIAM J. Matrix Anal. Appl. \textbf{36}(2), 549--579 (2015).
\newblock \doi{10.1137/130947830}

\bibitem{Gugercin2008}
Gugercin, S., Antoulas, A., Beattie, C.: $\mathcal{H}_2$ model reduction for
  large-scale linear dynamical systems.
\newblock SIAM J. Matrix Anal. Appl. \textbf{30}(2), 609--638 (2008).
\newblock \doi{10.1137/060666123}

\bibitem{Hartmann:2013:Balanced}
Hartmann, C., Sch{\"a}fer-Bung, B., Th{\"o}ns-Zueva, A.: Balanced averaging of
  bilinear systems with applications to stochastic control.
\newblock SIAM J. Control Optim. \textbf{51}(3), 2356--2378 (2013).
\newblock \doi{10.1137/100796844}

\bibitem{Horn:1991:MATAN2}
Horn, R., Johnson, C.: Topics in Matrix Analysis.
\newblock Cambridge Univ. Press, Cambridge, UK (1991)

\bibitem{jarlebring2017krylov}
Jarlebring, E., Mele, G., Palitta, D., Ringh, E.: Krylov methods for low-rank
  commuting generalized {S}ylvester equations.
\newblock Numer. Linear Algebra Appl.  (2018).
\newblock \doi{10.1002/nla.2176}.
\newblock (in press) e2176 nla.2176

\bibitem{Kressner:2015:Truncated}
Kressner, D., Sirković, P.: Truncated low-rank methods for solving general
  linear matrix equations.
\newblock Numer. Linear Algebra Appl. \textbf{22}(3), 564--583 (2015).
\newblock \doi{10.1002/nla.1973}

\bibitem{kressner2010krylov}
Kressner, D., Tobler, C.: {Krylov} subspace methods for linear systems with
  tensor product structure.
\newblock SIAM J. Matrix Anal. Appl. \textbf{31}(4), 1688--1714 (2010)

\bibitem{Mohler:1980:AnOverview}
Mohler, R.R., Kolodziej, W.J.: An overview of bilinear system theory and
  applications.
\newblock IEEE Transactions on Systems, Man, and Cybernetics \textbf{10}(10),
  683--688 (1980).
\newblock \doi{10.1109/TSMC.1980.4308378}

\bibitem{Neudecker:1992:AMatrixTrace}
Neudecker", H.: A matrix trace inequality.
\newblock J. Math. Anal. Appl. \textbf{166}(1), 302--303 (1992).
\newblock \doi{https://doi.org/10.1016/0022-247X(92)90344-D}

\bibitem{Shaker:2015:Control}
Shaker, H.R., Tahavori, M.: Control configuration selection for bilinear
  systems via generalised {H}ankel interaction index array.
\newblock Internat. J. Control \textbf{88}(1), 30--37 (2015).
\newblock \doi{10.1080/00207179.2014.938250}.
\newblock \urlprefix\url{https://doi.org/10.1080/00207179.2014.938250}

\bibitem{Shank2016}
Shank, S.D., Simoncini, V., Szyld, D.B.: Efficient low-rank solution of
  generalized {L}yapunov equations.
\newblock Numer. Math. \textbf{134}(2), 327--342 (2016)

\bibitem{Simoncini:2016:Computational}
Simoncini, V.: Computational methods for linear matrix equations.
\newblock SIAM Rev. \textbf{58}(3), 377--441 (2016)

\bibitem{Smith:1968:Matrix}
Smith, R.: Matrix equation {$XA + BX = C$}.
\newblock SIAM J. Appl. Math. \textbf{16}(1), 198--201 (1968).
\newblock \doi{10.1137/0116017}

\bibitem{vandereycken2010riemannian}
Vandereycken, B., Vandewalle, S.: A {Riemannian} optimization approach for
  computing low-rank solutions of {Lyapunov} equations.
\newblock SIAM J. Matrix Anal. Appl. \textbf{31}(5), 2553--2579 (2010)

\bibitem{ZHANG2002}
Zhang, L., Lam, J.: On {$H_2$} model reduction of bilinear systems.
\newblock Automatica J. IFAC \textbf{38}(2), 205 -- 216 (2002).
\newblock \doi{https://doi.org/10.1016/S0005-1098(01)00204-7}

\end{thebibliography}


\end{document}